\documentclass[reqno,12pt,letterpaper]{amsart}
\usepackage{amsmath,amssymb,amsthm,graphicx,mathrsfs,url,bbm,array,enumerate}
\usepackage[usenames,dvipsnames]{xcolor}
\usepackage[colorlinks=true,linkcolor=Red,citecolor=Green]{hyperref}
\usepackage{tikz-cd}
\usepackage{tikz}
\usepackage{pgfplots}
\usepackage{subfigure}
\usepackage{mathabx}
\usepackage{mathtools}

\usepackage[
    backend=biber,
    style=alphabetic,
    giveninits=true,
    maxalphanames=10,
    maxnames=10
]{biblatex}
\DeclareFieldFormat{pages}{#1}
\renewbibmacro{in:}{%
  \ifentrytype{article}
    {}
    {\bibstring{in}%
     \printunit{\intitlepunct}}}
\DeclareFieldFormat
  [article,inbook,incollection,inproceedings,patent,thesis,unpublished]
  {title}{\mkbibemph{#1}}
\DeclareFieldFormat{journaltitle}{#1\isdot}
\DeclareFieldFormat[article]{volume}{\mkbibbold{#1}}
\DeclareFieldFormat[article]{number}{\bibstring{number}\addnbspace #1}

\renewbibmacro*{journal+issuetitle}{%
  \usebibmacro{journal}%
  \setunit*{\addspace}%
  \iffieldundef{series}
    {}
    {\newunit
     \printfield{series}%
     \setunit{\addspace}}%
  \printfield{volume}%
  \setunit{\addspace}%
  \usebibmacro{issue+date}%
  \setunit{\addcomma\space}%
  \printfield{number}%
  \setunit{\addcolon\space}%
  \usebibmacro{issue}%
  \setunit{\addcomma\space}%
  \printfield{eid}
  \newunit}

 \DeclareFieldFormat*{title}{#1}
 \DeclareFieldFormat{title}{#1}
 \DeclareFieldFormat
   [article,inbook,incollection,inproceedings,patent,thesis,unpublished]
   {title}{\mkbibquote{#1\isdot}}
 \DeclareFieldFormat
   [suppbook,suppcollection,suppperiodical]
   {title}{#1}

\newcolumntype{L}{>{$}l<{$}}

\def\?[#1]{\textbf{[#1]}\marginpar{\Large{\textbf{??}}}}
\newtheorem{thm}{Theorem}
\newtheorem*{thm*}{Theorem}
\newtheorem{prop}{Proposition}

\newtheorem{lem}[prop]{Lemma}
\newtheorem{cor}[prop]{Corollary}
\newtheorem{rem}[prop]{Remark}

\numberwithin{equation}{section}
\numberwithin{prop}{section}

\theoremstyle{definition}

\newtheorem{notation}{Notation}

\DeclareMathOperator{\supp}{supp}

\DeclareMathOperator{\Tr}{{\rm tr}}
\DeclareMathOperator{\Vol}{Vol}

\DeclareMathOperator{\injrad}{InjRad}

\newcommand{\Hom}{{\rm Hom}}
\newcommand{\id}{{\rm id}}

\newcommand{\tr}{\mathrm{tr}}

\newcommand{\bbE}{\mathbb E}

\newcommand{\bbH}{\mathbb H}

\newcommand{\bbN}{\mathbb N}

\newcommand{\bbR}{\mathbb R}

\newcommand{\bbX}{\mathbb X}

\newcommand{\bbZ}{\mathbb Z}

\newcommand{\calP}{\mathcal P}

\newcommand{\calR}{\mathcal R}

\definecolor{green}{rgb}{0,0.8,0}

\begin{document}

\title[Polynomial bounds for eigenfunctions and eigenvalues on random covers]{Polynomial bounds for eigenfunctions and eigenvalues on random covers of hyperbolic surfaces}

\author{Elena Kim}
\address[Elena Kim]{Harvard University, Center of Mathematical Sciences and Applications, Cambridge, MA 02138, USA}
\email{elenakim@fas.harvard.edu}

\author{Zhongkai Tao}
\address[Zhongkai Tao]{Institut des Hautes \'{E}tudes Scientifiques, 91440 Bures-sur-Yvette, France}
\email{ztao@ihes.fr}

\begin{abstract}
Let $X$ be a compact connected orientable hyperbolic surface and let $X_n$ be a degree $n$ cover, taken uniformly at random. We show that, with high probability, the $L^\infty$ norm of every Laplace eigenfunction on $X_n$ with bounded eigenvalue decays polynomially in $n$. 
This gives a polynomial decay analogue of the logarithmic bound of \cite{GLMST} in the Weil--Petersson model.
Using similar methods,  we also show that, with high probability, the distribution of eigenvalues of the Laplacian on $X_n$ converges to the spectral measure of the hyperbolic plane with polynomially decaying error. 
Our proof relies on the Selberg pre-trace formula and a variant of the polynomial method.
\end{abstract}

\maketitle

\section{Introduction}

Let $X$ be a compact, connected, orientable hyperbolic surface. As the geodesic flow on such a surface is classically chaotic, we expect the asymptotic delocalization of its Laplace eigenfunctions on $X$, an indicator of quantum chaos.
More precisely, let $\Delta_X$ be the (non-negative) Laplace operator on $X$ and $u_j$ be the $j$-th eigenfunction with eigenvalue $\lambda_j\geq 0$:
\begin{equation*}
    \Delta_X u_j =\lambda_j u_j,\quad \|u_j\|_{L^2}=1.
\end{equation*}
Then on compact hyperbolic surfaces, we expect an improvement  to H\"ormander's \cite{Hor68}  $L^\infty$-bound for general compact surfaces: 
\begin{equation*}\label{eq:hormander}
    \|u_j\|_{L^{\infty}(X)}\leq C(1+\lambda_j)^{1/4}\|u_j\|_{L^2(X)}.
\end{equation*}

In the high frequency limit $\lambda_j \rightarrow \infty$, B\'erard \cite{Be77} obtained a  logarithmic improvement  for a family of surfaces which includes hyperbolic surfaces. Specifically, he proved the bound
\begin{equation}\label{eq:Berard}
    \|u_j\|_{L^\infty(X)}\leq C(1+\lambda_j)^{1/4}(\log(2+\lambda_j))^{-1/2}\|u_j\|_{L^2(X)}.
\end{equation}
In this paper, we study the high genus limit, rather than the high-frequency limit. Instead of  control in $\lambda$, we take a random family of hyperbolic surfaces with genus $g$ going to infinity, and ask for control in $g$.

The only previous result on $L^\infty$-bounds in the high genus limit is due to \cite{GLMST}, studying the Weil--Petersson model. The authors show that with high probability with respect to the Weil--Petersson measure,  eigenfunctions on a hyperbolic surface of genus $g$ satisfy an  estimate with logarithmic decay in $g$:
\begin{equation}\label{eq:GLMST-bound}
    \|u_j\|_{L^{\infty}(X)}\leq C(\lambda) (\injrad(X) \log g)^{-1/2}\|u_j\|_{L^2(X)}.
\end{equation}
Here, $\injrad(X)$ is the injectivity radius of $X$, i.e.,
 half the length of the shortest closed geodesic on $X$.
 
In this paper, we consider the \emph{random cover model}. Namely, we fix a hyperbolic surface $X$, and take a degree $n$ cover uniformly at random. All probabilities in this paper are with respect to the choice of random surface.  For the random cover model, we obtain a polynomial improvement over the  Weil--Petersson result ~\eqref{eq:GLMST-bound}.

\begin{thm}\label{thm:eigenfn}
    Let $X$ be a compact connected orientable hyperbolic surface of genus $g_0\geq 2$ and let $\varepsilon>0$. There exist $\alpha=\alpha(g_0, \varepsilon)>0$ and $C=C(X)>0$ such that the following is true. Let $X_n$ be a degree $n$ cover of $X$ taken uniformly at random and let $u_j$ be the $j$-th normalized eigenfunction of $\Delta_{X_n}$:
    \begin{equation*}
        \Delta_{X_n} u_j = \lambda_j(X_n) u_j,\quad \|u_j\|_{L^2(X_n)}=1.
    \end{equation*}
    Then with probability at least $1-n^{-1/10}$, for every $\Lambda\geq 1/4$ and every $\lambda_j(X_n)\leq \Lambda$, we have
    \begin{equation}\label{eq:linfty}
        \|u_j\|_{L^{\infty}(X_n)}\leq C \Lambda^{\frac{1}{4}+\varepsilon}n^{-\alpha}\|u_j\|_{L^2(X_n)}.
    \end{equation}
\end{thm}

Setting $\Lambda =\lambda_j$, B\'{e}rard's bound \eqref{eq:Berard} beats our result as $\lambda_j \rightarrow \infty$. However, for fixed $\Lambda$, our result obtains polynomial decay in the large cover limit. 

What do we expect to be the \emph{optimal} exponent for \eqref{eq:linfty}? To answer this question, we turn to $d$-regular graphs. Regular graphs with the graph Laplacian are a popular discrete model for the spectral theory of hyperbolic surfaces. 
In \cite[Theorem 1.4]{HYspec}, Huang and Yau  prove that with high probability,  an eigenvector on a random regular graph $G$ on  $N$ vertices is \emph{optimally delocalized}:
\begin{equation}\label{eq:HY1.4}
    \|u\|_{\ell^{\infty}}\leq C\frac{(\log N)^C}{\sqrt{N}}\|u\|_{\ell^2}.
\end{equation}
In \cite{GLMST}, the authors conjectured that an analogous bound holds in the Weil--Petersson model. We conjecture that the appropriate analogue of~\eqref{eq:HY1.4} also holds in the random cover model, i.e., the optimal value of $\alpha$ in \eqref{eq:linfty} is $\tfrac{1}{2}-\varepsilon$. 
Note that for a degree $n$ cover $X_n$,  $\Vol(X_n) = 4 \pi n(g_0-1)$. Therefore we have the trivial lower bound
$$\|u_j\|_{L^\infty(X_n)} \gtrsim n^{-\frac{1}{2}} \|u_j\|_{L^2(X_n)}.$$

Using similar proof methods to Theorem~\ref{thm:eigenfn}, we also prove a polynomial estimate for the distribution of eigenvalues on random covers. Our result improves earlier work of Monk \cite{monk22}, who proved a logarithmic bound in the Weil--Petersson setting.

\begin{thm}\label{thm:cover}
Let $X$ be a compact connected orientable hyperbolic surface of genus $g_0\geq 2$. For any $\varepsilon>0$, there exist $\alpha=\alpha(g_0,\varepsilon)>0$ and $C=C(X)>0$ such that the following is true. Let $X_n$ be a degree $n$ cover of $X$ taken uniformly at random and  let $\lambda_j(X_n)$ be the $j$-th Laplacian eigenvalue on $X_n$. Then with probability at least $1-n^{-1/10}$, for every $\Lambda\in [1/4,\infty)$ and every $\lambda_j(X_n)\in [1/4,\Lambda]$, we have
     \begin{equation}\label{eq:compare}
         |\lambda_j(X_n)-\lambda_j|\leq C \Lambda^{1/2+\varepsilon} n^{-2\alpha/3},
     \end{equation}
     where $\lambda_j\geq 1/4$ is defined by \begin{equation*}\label{eq:lai}
        \int_{0}^{\sqrt{\lambda_j-1/4}} r\tanh\left(\pi  r\right) dr =\frac{j}{n(2g_0-2)}.
    \end{equation*}  
Moreover, we have the following Weyl law for $N_{X_n}(\Lambda):=\#\{j:\lambda_j(X_n)\leq \Lambda\}$:
    \begin{equation}\label{eq:weyl}
      N_{X_n}(\Lambda)=(2g_0-2)n\int_0^{\sqrt{\Lambda-1/4}} r \tanh(\pi r) dr + O_{X}(n^{1-\alpha}\Lambda^{1/2+\varepsilon}),\quad \Lambda\in [1/4,\infty).
    \end{equation}
\end{thm}
Since the $\lambda_j$'s are evenly distributed in $[1/4,\infty)$ according to the spectral measure of the hyperbolic plane $\bbH$, \eqref{eq:compare} shows that the eigenvalues $\lambda_j(X_n)$ are also evenly distributed up to an $O(n^{-2\alpha/3}\Lambda^{1/2+\varepsilon})$-oscillation. In particular, the multiplicity of eigenvalues is bounded by $Cn^{1-\alpha}\Lambda^{1/2+\varepsilon}$ for a typical random cover, see \cite{GLMST,monk22,LeMa,he2025short} for related results.

We note that Hide, Macera, and Thomas \cite[Theorem 1.1]{HMT25a} showed that there exists $b, c > 0$ depending only
on the genus of $X$ such that a uniformly random degree $n$ cover $X_n$ of $X$ has
\begin{equation*}\label{eq:first_eigenvalue_cover}
\lambda_1^{\operatorname{new}}(X_n) \geq \frac{1}{4} - cn^{-b}, 
\end{equation*}
with probability tending to $1$ as $n \rightarrow \infty$. Here, $\lambda_1^{\operatorname{new}}(X_n)$ denotes the smallest eigenvalue of $X_n$ that is not an eigenvalue of $X$, accounting for multiplicities. Moreover, from \eqref{eq:weyl}, it is easy to see that there are at most $Cn^{1-\alpha}$ eigenvalues of $\Delta_{X_n}$ below $1/4$. Therefore, we conclude the following.
\begin{cor}
The new eigenvalues of $X_n$ below $1/4$ also satisfy \eqref{eq:compare} with $2\alpha/3$ replaced by $\min(2\alpha/3,b)$ and probability tending to $1$ as $n\to \infty$. 
\end{cor}

We note that Theorem~\ref{thm:cover} does not describe how the eigenvalues oscillate at finer scales as in \cite{Rudnick23,Naud25}. 
Rudnick and Naud showed that in the Weil--Petersson model and random cover model, respectively, the number variance of the eigenvalues in the large genus and small window limit converges to that of the GOE eigenvalues. 
With the polynomial method, one can take the size of the window $L^{-1}$ to depend on the degree of the cover: $L=n^{\alpha}$.
Indeed, this idea is reminiscent of the proof of Theorem~\ref{thm:eigenfn}, in which there is a term \eqref{eq:loc-osci} that describes the local oscillation of eigenfunctions away from the local Weyl law.

\subsection{Framework for general random hyperbolic surfaces}
As mentioned above, compact random covers  are not the only way to study the large genus limit of random hyperbolic surfaces. As we will explain in \S\ref{subsection:previous_work}, there has been much recent progress in the spectral theory of Weil--Petersson random surfaces and Brooks--Makover surfaces.

In this subsection, we generalize bounds of the form~\eqref{eq:linfty} and ~\eqref{eq:weyl} (from Theorem~\ref{thm:eigenfn} and Theorem~\ref{thm:cover}, respectively) to other random models beyond random covers. The proof of such bounds follows from
bounding the expectations that appear in~\eqref{eq:expan-g-1} and~\eqref{eq:eigenvalue_condition}. 
To bound these expectations, we use the polynomial method pioneered by \cite{CGVTvH24} and generalized by \cite{MPvH25}. This argument crucially relies on polynomial approximations of our expectations. The proofs of these polynomial approximations depend on the specifics of the choice of random model. This is in contrast to the rest of the argument: the theorems in this subsection show that assuming a random model has the appropriate  polynomial approximations, we conclude results of the form Theorem~\ref{thm:eigenfn} and Theorem~\ref{thm:cover}.

In this paper, we first prove the general theorems in this subsection, then show that Theorem~\ref{thm:eigenfn} and Theorem~\ref{thm:cover} follow as corollaries. 
In a forthcoming paper,  we plan to adapt the ideas of Theorem~\ref{thm:eigenfn} and Theorem~\ref{thm:cover} to the Weil--Petersson model.

Now we introduce some notation. First, let $dz \coloneqq y^{-2}dxdy$ be the hyperbolic volume form on $\bbH$.
For a continuous function $F: (0,\infty) \to \mathbb{R}$, let
\begin{equation*}
    \|F\|_{I}:=\sup_{x\in I}|F(x)|.
\end{equation*}
For a hyperbolic surface $X$, a \emph{geodesic loop} $\gamma=\gamma(z)$ based at $z\in X$ is a geodesic segment with both start and end at $z$. It gives a representative of a nontrivial element in $\pi_1(X,z)$. We denote its length by $\ell_{\gamma}(X,z)$. Similarly, we denote the length of a closed geodesic $\gamma$ on $X$ by $\ell_{\gamma}(X)$. Every closed oriented geodesic $\gamma$ in $X$ determines a nontrivial conjugacy class $[\tilde{\gamma}] \subset \Gamma \setminus \{\id\}$. By abuse of notation, we use $\gamma$ to denote both geodesic loops and closed geodesics  and their associated elements of $\pi_1(X, z)$ and $\pi_1(X)$.

We define $\injrad_z(X) \coloneqq \sup \{r \geq 0 :  B_r(z) \subset X \text{ is isometric to a ball in } \bbH\}$ to be the injectivity radius at a point $z \in X$. As mentioned above, we define $\injrad(X) \coloneqq \inf\{\injrad_z(X): z \in X\}$ to be the injectivity radius of $X$. 
For $\delta_0 \geq 0$ define
$$X_{\delta_0} \coloneqq X \setminus \{z \in X: \injrad_z(X) < \delta_0\}.$$

The first theorem we present in this subsection shows under which general conditions we achieve polynomial bounds on eigenfunctions. In the case where our family of random hyperbolic surfaces is  random covers of a base surface $X$ of genus $g_0$, set:
\begin{itemize}
    \item $g =n(g_0-1) +1$, where  $n$ is the degree of the cover;
    \item  $\delta_0 = \min\{\injrad(X)/2, \mathrm{arcsinh}(1)/4\}$.
\end{itemize}
\begin{thm}\label{thm:eigen-g}
    Suppose we have a random family of compact oriented (not necessarily connected) hyperbolic surfaces $Y = \Gamma \backslash \bbH$ of area $2\pi(2g-2)$. Suppose there is an asymptotic parameter $n=n(g)\to \infty$ as $g \rightarrow \infty$ such that the following holds for some $C_0>2$, $\kappa>2$, and $0<\delta_0<\frac{1}{2}\mathrm{arcsinh}(1) $:
    \begin{enumerate}
     \item \emph{$n$ and $g$ are proportional:}
     \begin{equation}\label{g-n-1}
    C_0^{-1} n \leq g \leq C_0 n.
     \end{equation}
        \item \emph{Off-diagonal polynomial approximation:} Define
        \begin{equation}\label{eq:ttilde-gamma-def}
    \widetilde{\Gamma}^4:=\left(\Gamma\setminus \{\id\}\right)^4\setminus \left\{(\gamma^{k_1},\gamma^{k_2},\gamma^{k_3},\gamma^{k_4}):\gamma\in  \Gamma\setminus \{\id\},\,\, k_1,k_2,k_3,k_4\in \bbZ\setminus\{0\}\right\}.
\end{equation}
        For  $L\geq 1$ and continuous functions $F_1,\ldots,F_8:[0,\infty)\to \mathbb{R}$ with $\supp F_\ell \subset [0, L]$,  there exist $p,p_1\leq C_0L$ and $f_j\in \mathbb{R}$ independent of $n$ such that
        \begin{equation}\label{eq:expan-g-1}
            \begin{split}
            &\left|\mathbb{E}\left[\int_{Y_{\delta_0}}  \sum_{(\gamma_1, \ldots, \gamma_8) \in \left(\tilde{\Gamma}^4\right)^2}   F_1(\ell_{\gamma_1}(Y,z))\cdots F_8(\ell_{\gamma_8}(Y,z)) dz\right] -\frac{1}{Q_L(1/n)}\sum_{j=1}^{p_1} f_j n^{-j} \right| \\
            &\leq (C_0 p)^{\kappa p}n^{-p-1}\prod_{\ell=1}^{8}\|F_{\ell}\|_{(0,L]},\quad  n\geq (C_0 p)^{\kappa},
            \end{split}
        \end{equation}
        where $Q_L(1/n)\in [C_0^{-1},C_0]$ for $n\geq (C_0 p)^{\kappa}$.
    \end{enumerate}
    Then there exist $\alpha=\frac{1}{32(2\kappa+1)}>0$ and $C=C(\kappa,C_0)>0$ such that with probability at least $1-n^{-1/10}$, for every $\Lambda\geq 1/4$ and every $\lambda_j(Y)\leq \Lambda$, we have
    \begin{equation}\label{eq:general-linfty}
        \|u_j\|_{L^\infty(Y_{2\delta_0})}\leq C\Lambda^{2} \delta_0^{-1/2} n^{-\alpha}\|u_j\|_{L^2(Y)}.
    \end{equation}
\end{thm}
Note the dependence on $\delta_0$ in~\eqref{eq:general-linfty} mirrors the dependence on the injectivity radius in~\eqref{eq:GLMST-bound}.

Similarly, we have the following theorem for polynomial bounds on eigenvalues. 
\begin{thm}\label{thm:lambda-g}
    Suppose we have a random family of compact oriented (not necessarily connected) hyperbolic surfaces $Y = \Gamma \backslash \bbH$ of area $2\pi(2g-2)$. Suppose that there is an asymptotic parameter $n=n(g) \to \infty$ as $g \rightarrow \infty$ such that the following holds with $C_0>2$ and $\kappa>2$:

 \begin{enumerate}
    \item \emph{$n$ and $g$ are proportional:}  
    \begin{equation}\label{eq:g-n-2}
    C_0^{-1} n \leq g \leq C_0 n.
    \end{equation}
    \item \emph{Injectivity radius bound:} Let $\injrad_1(Y):=\min\{1,\injrad(Y)\}$. We have  \begin{equation}\label{eq:sys-2}
        \mathbb{E}[|\log\injrad_1(Y)|^2]<C_0.
    \end{equation}
 \item  \emph{Polynomial approximation:} For  $L\geq 1$ and continuous function $F:(0,\infty)\to \mathbb{R}$ with $\supp F \subset [0, L]$,  there exist $p,p_1\leq C_0 L$ and $f_j\in \mathbb{R}$ independent of $n$ such that
        \begin{equation}\label{eq:eigenvalue_condition}
            \begin{split}
            &\left|\mathbb{E}\left[\sum_{\gamma_1, \gamma_2 \in \calP(Y)} \sum_{k_1, k_2 =1}^\infty \frac{\ell_{\gamma_1}(Y)}{2\sinh \frac{k_1\ell_{\gamma_1}(Y)}{2}}\frac{\ell_{\gamma_2}(Y)}{2\sinh \frac{k_{2}\ell_{\gamma_2}(Y)}{2}}F(k_{1}\ell_{\gamma_1}(Y)) F(k_{2}\ell_{\gamma_2}(Y))\right]-\frac{1}{Q_L(1/n)}\sum_{j=0}^{p_1} f_j n^{-j}\right|\\
            &\leq (C_0p)^{\kappa p} n^{-p}\|F\|_{(0,L]}^2 \quad n\geq (C_0p)^\kappa,
            \end{split}
        \end{equation}
        where $\calP(Y)$ denotes the set of primitive closed geodesics on $Y$ and $Q_L(1/n)\in [C_0^{-1},C_0]$ for $n\geq (C_0p)^{\kappa}$.
\end{enumerate}
Then there exist $\alpha=\frac{1}{4(\kappa+1)}>0$ and $C=C(\kappa,C_0)>0$ such that with probability at least $1-n^{-1/10}$, for $N_Y(\Lambda):=\#\{\lambda_j(Y)\leq \Lambda\}$, $\Lambda \geq 1/4$, we have
\begin{equation}\label{eq:weyl-g}
    \left|N_Y(\Lambda)-\frac{\Vol(Y)}{2\pi} \int_0^{\sqrt{\Lambda-1/4}} r\tanh (\pi r) dr\right|\leq C n^{1-\alpha} \Lambda^2 , \quad \Lambda\in [1/4,\infty).
\end{equation}
\end{thm}

We remark that the naive analogous condition to~\eqref{eq:sys-2} for Theorem~\ref{thm:eigen-g} to hold is a polynomial, rather than logarithmic, bound on the injectivity radius.  As such a condition does not hold in the Weil--Petersson setting, we instead prove an $L^\infty$ bound that holds on $Y_{\delta_0}$.

As we explain in \S \ref{subsection:proof_idea}, to prove Theorem~\ref{thm:eigen-g} we need to take the square of a fourth power, thus requiring 8 factors of $F_{\ell}(\ell_{\gamma_i}(Y, z))$ in the product~\eqref{eq:expan-g-1}, instead of two factors in \eqref{eq:eigenvalue_condition}.

\subsection{Previous work}\label{subsection:previous_work}\textcolor{white}{.}

\noindent \emph{Eigenfunction bounds.} As mentioned above, in the general setting of compact surfaces,  H\"ormander \cite{Hor68} obtained the following sharp eigenfunction bound:
\begin{equation}\label{eq:Hormander}
    \|u_j\|_{L^{\infty}}\leq C(1+\lambda_j)^{\frac{1}{4}}\|u_j\|_{L^2}. 
\end{equation}
Sogge \cite{Sog88} generalized the $L^\infty$ bound to  the following  $L^p$ bounds:
\begin{equation}\label{eq:sog-1}
   \|u_j\|_{L^{p}}\leq C(1+\lambda_j)^{\sigma(p)}\|u_j\|_{L^2},\quad \sigma(p)=\left\{\begin{array}{ll}
   1/4-1/p,    & 6\leq p\leq \infty,  \\
   \frac{1}{4}\left(\frac{1}{2}-\frac{1}{p}\right),    & 2\leq p\leq 6.
    \end{array}\right. 
\end{equation}
Any improvement on~\eqref{eq:Hormander} or~\eqref{eq:sog-1} relies on the restriction to a setting with additional assumptions. 
We recall that in the classically chaotic regime of negatively curved surfaces (where we expect quantum chaos in the form of eigenfunction delocalization), B\'erard \cite{Be77} obtained a logarithmic improvement. 
Polynomial improvements over the H\"ormander bound are much harder and presently known only in settings with even more structure. One such source of structure is arithmetic: for compact arithmetic surfaces, strong bounds were obtained by Iwaniec and Sarnak \cite{iwsa} in the spectral aspect and by Hu and Saha \cite{saha2020, husaha2020} in the depth aspect. See also Ki \cite{ki2023}, Assing and Toma \cite{assing2025}, and Fischer \cite{fischer2026} for recent developments on non-compact arithmetic surfaces. 
Iwaniec and Sarnak \cite{iwsa,Sar95,Sar03} further conjectured  that on hyperbolic surfaces,
$$\|u_{j}\|_p \ll_{\varepsilon, p} \lambda_j^\varepsilon \|u_{_j}\|_2 \quad \text{for } 2 \le p \le \infty,$$
reflecting the expectation that eigenfunctions on chaotic surfaces should be as delocalized as possible. This expectation is partially motivated by Berry's random wave conjecture \cite{Be77a}, which predicts that high-frequency eigenfunctions on chaotic systems behave like random waves; bounds on $L^p$ norms control the moments of the resulting limiting distribution.

Polynomial improvements over the H\"ormander bound have also been obtained by exploiting randomness, rather than arithmetic structure. Ingremeau and Vogel \cite{IV24} proved polynomial improvement over the H\"ormander bound by taking random pseudodifferential perturbations of the Laplacian on a fixed negatively curved manifold. Our result, Theorem~\ref{thm:eigenfn}, provides a polynomial improvement in a different random regime: we work with the unperturbed Laplacian on a \emph{random hyperbolic surface}, fix a bounded eigenvalue range, and take the degree of the cover to infinity. Prior work studying eigenfunctions in the high-genus limit focused on the Weil--Petersson model of random hyperbolic surfaces. Gilmore, Le Masson, Sahlsten, and Thomas \cite{GLMST} proved a logarithmically improved $L^p$ eigenfunction estimate for Weil--Petersson random hyperbolic surfaces. Thomas \cite{Thomas22} proved an $L^2$-mass delocalization result for eigenfunctions on Weil--Petersson random hyperbolic surfaces with large genus.

Under Benjamini--Schramm convergence, Le Masson and Sahlsten \cite{LMS} proved a quantum ergodicity statement.  See also \cite{LMS24} for the case of Eisenstein series and Hippi \cite{Hippi} for a quantum mixing result. 

\noindent \emph{Eigenvalue bounds.} We now turn to prior results on the low-energy spectral theory of random hyperbolic surfaces.  Most of such work focused on the   \emph{spectral gap}, i.e., the smallest eigenvalue on a hyperbolic surface excluding the trivial eigenvalue at $0$. 
The spectral gap appears when studying the error term for counting closed geodesics of bounded length, the error term for the hyperbolic lattice counting problem, and the error term for the rate of mixing of the geodesic flow. 
The spectral gap on hyperbolic surfaces can be thought of as an analog of the Alon--Boppana bound on graphs. For additional context on spectral gaps, see the survey \cite{MN26}. 

As the hyperbolic plane is the universal cover of a hyperbolic surface, we predict that with high probability, in the large genus limit, the spectrum of large random hyperbolic surfaces resembles that of the hyperbolic plane. As $\bbH$ has continuous spectrum in $[\tfrac{1}{4}, \infty)$, we expect a spectral gap of $\tfrac{1}{4}$.

The first result on spectral gaps is due to Huber~\cite{Hu74}, who proved that for any sequence of compact hyperbolic surfaces $X_n$ with genera $g(X_n)$ tending to infinity, we have $\limsup_{n \rightarrow \infty} \lambda_1(X_n) \leq \tfrac{1}{4}$.  More recently,  Magee, Naud, and Puder~\cite{MNP22} showed that for all $\varepsilon>0$, with high probability, a degree $n$ cover of a compact hyperbolic surface has no new eigenvalues below $\tfrac{3}{16} - \varepsilon$.
In  \cite{MPvH25}
Magee, Puder,  and van Handel improved $\tfrac{3}{16}-\varepsilon$ to $\tfrac{1}{4}-\varepsilon$ using the \emph{polynomial method}.

The polynomial method was introduced by Chen, Garza-Vargas, Tropp, and van Handel \cite{CGVTvH24} to prove strong convergence of random matrices. The paper \cite{MPvH25} extended the method beyond the original random matrix setting to random permutation representations of the surface group.
Combining their result with the Selberg trace formula, Hide, Macera, and Thomas  \cite{HMT25a} obtained a spectral gap result with polynomial error. Specifically, they showed $\lambda_1(X_n) \geq \frac{1}{4}-O(n^{-b})$ for uniformly random degree $n$ covers $X_n$  of a closed hyperbolic surface $X$ with high probability. In all of these applications, the polynomial method is paired with a spectrally-averaged identity---the Selberg
trace formula or an operator-norm bound---which provides
control of spectral averages.  To prove Theorem~\ref{thm:eigenfn}, we instead combine the polynomial method with the Selberg pre-trace formula, a pointwise spectral identity. 

The work of \cite{MNP22} was generalized to surfaces with variable negative curvature in \cite{HMN25} by Hide, Moy, and Naud. In the even more general case of covers of closed Riemannian surfaces with Anosov geodesic flow, Moy \cite{Moy26} proved the existence of a spectral gap for Pollicott--Ruelle resonances. In constant curvature, a spectral gap for the Laplacian implies a spectral gap for the Pollicott--Ruelle spectrum. However, this correspondence does not hold in variable curvature.

Much recent progress has also been made in the regime of spectral gaps for covers of finite-area noncompact hyperbolic surfaces $X$.  Hide and Magee studied this model in  \cite{HM23}. They showed for any $\varepsilon>0$, with high probability, a uniformly random degree $n$  cover of $X$ has no new eigenvalues below $\tfrac{1}{4} - \varepsilon$.
In \cite{Hi24}, Hide strengthened the work of  \cite{HM23} to replace $\varepsilon$ with $c \tfrac{(\log \log \log n)^2}{\log \log n}$.
In \cite{Mo25}, Moy studied spectral gaps in the case of covers of noncompact, geometrically finite surfaces with pinched sectional curvature. In \cite{ballmann2025spectral}, Ballmann, Mondal, and Polymerakis studied random covers of a complete connected Riemannian manifold with Ricci curvature bounded from below, under certain conditions of the fundamental group.

Finally, we discuss progress towards understanding the Laplace spectrum for other models of random large genus hyperbolic surfaces. For Weil--Petersson random surfaces, Wu and Xue \cite{Wuxue} and Lipnowski and Wright \cite{LipWr} showed that with high probability, $\lambda_1(X)\geq \frac{3}{16}-o(1)$. 
This bound was improved in the series of works \cite{AM23}, \cite{AM24:1}, \cite{AM24:2}, \cite{AM25} by Anantharaman and Monk. Specifically, they concluded that for all $\varepsilon>0$, $\lambda_1(X)\geq \tfrac{1}{4} - \varepsilon$ for a Weil--Petersson random hyperbolic surface with high probability. Recently, Hide,  Macera, and Thomas \cite{HMT25b} used the polynomial method developed in \cite{CGVTvH24,MPvH25} to give a polynomial error term in the result of Anantharaman and Monk. They show that there exists a constant $c>0$ such that a genus $g$ closed hyperbolic surface satisfies $\lambda_1(X)\geq \tfrac{1}{4} - O(\tfrac{1}{g^c})$ with high probability. In \cite{HWX26}, He, Wu, and Xue show a uniform spectral gap for hyperbolic surfaces with genus $g$ and $n= O(g^\alpha)$ cusps for $\alpha \in [0, \tfrac{1}{2})$. For the Brooks--Makover model, Shen and Wu \cite{shen2025nearly} proved the spectral gap $1/4-c/\log n$ using the polynomial method.

We conclude by mentioning work studying fluctuations of eigenvalues in short energy windows. In \cite{Rudnick23}, Rudnick studied the variance of the spectral statistics in the Weil--Petersson model. Taking the large genus limit, then the short window limit, the variance converges to that of GOE statistics. See also Rudnick and Wigman \cite{RuWi}. Fluctuations in the random cover model were studied by Naud \cite{Naud25} and Maoz \cite{Mao23}. These results were generalized to random covers of negatively curved surfaces by Moy \cite{Moy24}.

\subsection{Proof idea} 
\label{subsection:proof_idea} 
Here, we outline the proof of the eigenfunction estimate Theorem~\ref{thm:eigen-g}. As discussed above, Theorem~\ref{thm:eigenfn} follows by showing that the random cover model satisfies the conditions of Theorem~\ref{thm:eigen-g}.  Theorem~\ref{thm:lambda-g} (which we use to prove Theorem~\ref{thm:cover}) follows from the same main ideas, however Theorem~\ref{thm:eigen-g} has some extra difficulties. We highlight these difficulties as they arise. 
In \S\ref{subsection:structure-of-paper}, we list where each step of the proofs appears in the paper.

We begin by constructing a test function $h \circ f_{\Lambda_0}$ that approximates the indicator function of an interval of eigenvalues up to scale $ 
(1+\Lambda) q^{-1}$ and cuts near $\Lambda \geq 0$. In addition, the function $f_{\Lambda_0}$  is sufficiently nice so that we can apply the Selberg (pre-)trace formula and the function $h$ is a polynomial of fixed degree $q$ so we may apply the polynomial method.  

We now recall the Selberg pre-trace formula (stated in full in Lemma~\ref{lem:pretrace-twist}):
\begin{equation}\label{eq:pretrace-intro}
    \sum_{j} (h\circ f_{\Lambda_0})(\sqrt{\lambda_j-1/4}) |u_j(z)|^2 - \frac{1}{2\pi}\int_0^{\infty}(h\circ f_{\Lambda_0})(r) r\tanh \pi r dr =\sum_{\gamma \in \Gamma\setminus \{\id\}} K_{(h\circ f_{\Lambda_0})^{\vee}}(z,  \gamma z).
\end{equation}

Thus,  to prove Theorem~\ref{thm:eigen-g}, it suffices to show that the right-hand side of~\eqref{eq:pretrace-intro} is small with high probability uniformly in $z\in Y$, where $Y$ is our random hyperbolic surface. In the proof of Theorem~\ref{thm:lambda-g}, we replace the Selberg pre-trace formula with the Selberg trace formula and bound the analogous quantity to the right-hand side of~\eqref{eq:pretrace-intro}. This is done  by estimating the expectation of its square, then applying Chebyshev's inequality. However, if we attempt the same method for Theorem~\ref{thm:eigen-g}, we conclude
\begin{equation*}\label{eq:into-ex2}
    \bbE\left[\left(\sum_{\gamma \in \Gamma\setminus \{\id\}}  K_{(h\circ f_{\Lambda_0})^{\vee}}(z,  \gamma z)\right)^2\right] \lesssim_q n^{-1}.
\end{equation*}
This implies~\eqref{eq:pretrace-intro} is small for a \emph{fixed} $z\in Y$. In order to get a uniform control over $z\in Y$, we instead take the fourth power of~\eqref{eq:pretrace-intro} and separate it into  a diagonal sum and an off-diagonal sum:
\begin{equation*}\label{eq:fourth_power}
    \begin{split}
    \left(\sum_{\gamma \in \Gamma\setminus \{\id\}}  K_{(h\circ f_{\Lambda_0})^{\vee}}(z,  \gamma z)\right)^4=S(z)+\sum_{(\gamma_1,\gamma_2,\gamma_3,\gamma_4)\in \widetilde{\Gamma}^4}\prod_{\ell=1}^{4}K_{(h\circ f_{\Lambda_0})^{\vee}}(z,\gamma_{\ell}z)
\end{split}
\end{equation*}
where
\begin{equation*}
     S(z) := \frac{1}{2}\sum_{k_1,k_2,k_3,k_4\in \bbZ\setminus \{0\}}\sum_{\gamma\in \Gamma_0}\prod_{\ell=1}^{4} K_{(h\circ f_{\Lambda_0})^{\vee}}(z,\gamma^{k_\ell}z)
\end{equation*}
is the diagonal part and the sum over $\tilde{\Gamma}^4$ (defined in \eqref{eq:ttilde-gamma-def}) 
is the off-diagonal part. 
Due to the rapid decay coming from the fourth power, taking the spectral cutoff $h \circ f_{\Lambda_0}$ to be narrow is sufficient to make the diagonal part $S(z)$ small. 

Now the off-diagonal part has extra cancellation since two of those geodesics generate a free group of rank two. This is captured in the assumption \eqref{eq:expan-g-1}. Using Chebyshev's inequality, Theorem~\ref{thm:eigen-g} follows from Proposition~\ref{prop:expect-fn}, i.e.,
\begin{align}\label{eq:intro-expect}
\bbE \left[ \int_{Y} \left( \sum_{(\gamma_1,\gamma_2,\gamma_3,\gamma_4)\in\widetilde{\Gamma}^4}\prod_{\ell=1}^{4}K_{(h\circ f_{\Lambda_0})^{\vee}}(z,\gamma_{\ell}z)\right)^2 d z \right] \leq C(h)\delta_0^{-8} \frac{\Lambda_0^{12} q^{4 \kappa}} {n}. 
\end{align}

We show \eqref{eq:intro-expect} using the generalized polynomial method of \cite{MPvH25}.  For more on the polynomial method, see  \cite{vH25}. 
The first step is  to use the `niceness' of the test function $h \circ f_{\Lambda_0}$ to prove a preliminary bound for the left-hand side of~\eqref{eq:intro-expect}. Then, using~\eqref{eq:expan-g-1}, we know that the left-hand side of ~\eqref{eq:intro-expect} 
is well-approximated by a polynomial $P(1/n)$ with $\mathrm{deg}\, P(x)\lesssim q$ and $P(0)=0$.
Note that our preliminary bound for $h \circ f_{\Lambda_0}$ also holds for the approximating polynomial $P(1/n)$. Thus, we apply the  Markov brothers' inequality \eqref{eq:markov} to $P(1/n)$ to bound $P(1/n)$ by polynomial decay in $n$. Transferring this bound to the left-hand side of~\eqref{eq:intro-expect}, we conclude the bound~\eqref{eq:intro-expect}.

Note that \eqref{eq:intro-expect} depends on the choice of $h$. This artifact of the polynomial method prevents an optimal exponent in Theorem~\ref{thm:eigenfn}.

Finally, to prove Theorem~\ref{thm:eigenfn}, it suffices to show that the random cover model satisfies the polynomial approximation~\eqref{eq:expan-g-1}. Adapting the argument of \cite{MPvH25}, we  use the structure of random covers to reduce the proof to a graph counting argument.

\subsection{Structure of the paper}\label{subsection:structure-of-paper}
\begin{itemize}
    \item In \S\ref{section:preliminaries}, we review the required preliminaries for this paper. First in \S\ref{subsection:surface_group}, we define the random cover model. In \S\ref{subsection:prelim-lemmas}, we recall some well-known lemmas used in our proofs. In \S\ref{subsection:test}, we construct our test function $h \circ f_{\Lambda_0}$.
    \item In \S\ref{section:proof-theorem3}, we prove Theorem~\ref{thm:eigen-g}. Specifically, in \S\ref{subsec:eigen-g} we prove Theorem~\ref{thm:eigen-g} under the assumption of   Proposition~\ref{prop:expect-fn}. Then in \S\ref{subsec:proof_31}, we prove Proposition~\ref{prop:expect-fn}. 
    \item In \S\ref{section:adapt_polynomial_method}, we prove Theorem~\ref{thm:lambda-g}. Specifically, in \S\ref{subsection:proof_of_main_thm} we prove Theorem~\ref{thm:lambda-g} under the assumption of   Proposition~\ref{prop:expectation_squared}. Then in \S\ref{subsection:proof_of_prop}, we prove Proposition~\ref{prop:expectation_squared}. 
    \item In \S\ref{section:polynomial_expansion}, we prove Theorems~\ref{thm:eigenfn} and~\ref{thm:cover} by showing the random cover model satisfies the conditions of Theorems~\ref{thm:eigen-g} and~\ref{thm:lambda-g}, respectively. Towards showing these conditions, we prove key lemmas in \S\ref{subsection:expand-trace}. Then in \S\ref{subsection:polynomial_eigenfunction}, we prove Theorem~\ref{thm:eigenfn}, along with a remark on isotropic delocalization. In \S\ref{subsubsection:polynomial_expansion}, we prove Theorem~\ref{thm:cover}.
\end{itemize}

\subsection*{Notation}
We say $A\lesssim B$ or $A=O(B)$ if there exists a constant $C>0$ such that $A\leq CB$. To emphasize the dependence of the constant on a parameter $a$, we write $\lesssim_a$ or $O_{a}(\cdot)$. In this paper, the constant $C$ usually depends on the base hyperbolic surface $X$, but does not depend on the degree of the cover, $n\in \bbN$. The value of $C$ may vary from line to line. We use $\tr A$ for the trace of a matrix or a linear map $A$. We do not use the normalized trace in this paper.

For $f\in \mathscr{S}(\bbR)$, we define its Fourier transform by
\begin{equation*}
    \hat{f}(\xi)=\int_{\bbR} f(x) e^{-ix\xi} dx
\end{equation*}
and its inverse Fourier transform by
\begin{equation*}
    \check{f}(x)=\frac{1}{2\pi}\int_{\bbR} f(\xi) e^{ix\xi} d \xi.
\end{equation*}
\subsection*{Acknowledgments}
We would like to thank Semyon Dyatlov (supported by NSF grant DMS-2400090), Davide Macera, Michael Magee, Julien Moy, Doron Puder, Qiuyu Ren, Nikhil Srivastava, Yuhao Xue, and all the participants of the discrete analysis seminar at Berkeley for many interesting discussions. We thank Yulin Gong for suggesting that our work could be used to prove a Weyl law and for proposing that we prove an eigenfunction estimate. Additionally, the interpolation arguments \eqref{eq:interpo-1} and \eqref{eq:unif-thm2} are due to Y. Gong. We thank Francis Nier for suggesting presenting Theorem~\ref{thm:eigen-g} and Theorem~\ref{thm:lambda-g} separately.
The work in  Remark~\ref{rem:exchange} followed from a suggestion of  Tuomas Sahlsten. 
EK is supported by NSF GRFP - 1745302 and NSF DMS - 2602417. EK is grateful for the hospitality of the Institut des Hautes \'Etudes Scientifiques, where some of this collaboration was conducted.

\section{Preliminaries}\label{section:preliminaries}
We recall the necessary preliminaries in this section. 

\subsection{Random covers}\label{subsection:surface_group}
We first define the random cover model. 

Fix $X = \Gamma \backslash \bbH$ to be a closed orientable surface of genus $g_0$.
Then $\Gamma$ is a surface group with $2g_0$ generators:
$$\Gamma = \langle a_1, a_2, \ldots, a_{2g_0-1}, a_{2g_0} | [a_1, a_2] \cdots [a_{2g_0-1}, a_{2g_0}] = 1 \rangle.$$

We now define the degree $n$ covers of $X$. 
Set
$$\bbX_{g_0, n} \coloneqq \Hom(\Gamma, S_n).$$
This set is finite and endowed with the uniform probability measure.
Let $\varphi_n: \Gamma \rightarrow S_n$ be a permutation representation, where $S_n$ is the symmetric group on $[n] \coloneqq \{1, \ldots, n\}$. The action of  $\Gamma$ on $\bbH \times [n]$ is given by
\begin{equation}\label{eq:cover-q}
    \gamma \cdot (z, i) = (\gamma z, \varphi_n(\gamma)(i)).
\end{equation}
Then 
$$X_n \coloneqq \Gamma   \backslash_{\varphi_n}(\bbH \times [n])$$
is a \emph{degree $n$ cover}.

Now define $V_n \coloneqq \ell^2([n])$. Denote by $\operatorname{std}_n$ the standard representation  of the symmetric group $S_n$  of permutation matrices. Note that $\operatorname{std}_n$ acts on $V_n$. We compose $\varphi_n$ and $\operatorname{std}_n$ to obtain a representation of $\Gamma$ on $V_n$:
\begin{equation*}
\label{eq:rho_def}
\rho_{\varphi_n} \coloneqq \operatorname{std}_n \circ \varphi_n : \Gamma \rightarrow \operatorname{End}(V_n).
\end{equation*}
Oftentimes, we will drop the $\varphi_n$ subscript and use $\rho = \rho_{\varphi_n}$.

\subsection{Preliminary lemmas}\label{subsection:prelim-lemmas}

In this section, we cite some well-known lemmas necessary to our later proofs. 

Let $X$ be a hyperbolic surface with $L^2$-normalized eigenfunctions $u_j$ with eigenvalue $\lambda_j$. 
Recall that every closed oriented geodesic $\gamma$ in $X$ determines a nontrivial conjugacy class $[\tilde{\gamma}] \subset \Gamma \setminus \{\id\}$. 
\begin{notation}
\label{notation:gamma}
By abuse of notation, we use $\gamma$ to denote both the geodesic and the element of the conjugacy class $[\tilde{\gamma}]$ with the shortest representing word in the
generators. 
\end{notation}

We have the Selberg pre-trace formula (see,  for example, \cite[Theorem 9.3.7]{Bu10}).

\begin{lem}[Selberg pre-trace formula]\label{lem:pretrace-twist}
Suppose $\phi(x) \in C_c^{\infty}(\mathbb{R})$ is an even function and let
\begin{equation}\label{eq:selberg-trans}
    k(t):=-\frac{1}{\sqrt{2}\pi}\int_t^{\infty}\frac{\phi'(s)}{\sqrt{\cosh s-\cosh t}} ds, \quad t\geq 0,\quad K(z,w):=k(d_{\bbH}(z,w)),
\end{equation}
where $d_{\mathbb{H}}(z,w)$ is the hyperbolic distance on $\bbH$.
Then
\begin{equation}\label{eq:pretrace-1}
    \sum_{j} \hat{\phi}\left(\sqrt{\lambda_j-1/4}\right) \left|u_j(z)\right|^2 = \frac{1}{2\pi}\int_0^{\infty}\hat{\phi}(r) r\tanh \pi r dr+\sum_{\gamma \in \Gamma\setminus \{\id\}}  K(z,  \gamma z).
\end{equation}
\end{lem}

In applications of \eqref{eq:pretrace-1}, we use the following notation.
\begin{notation}\label{notation:K}
Set $k=k_{\phi}$ and $K=K_{\phi}$ to be the $k$ and $K$ which come from $\phi$ as in \eqref{eq:selberg-trans}.
\end{notation}

Let $\calP(X)$ denote the set of primitive closed geodesics on $X$. We now recall the  Selberg trace formula  (see, for example, \cite[Theorem 3.4]{borth}). 
\begin{lem}[Selberg trace formula]
Suppose $\phi(x) \in C_c^{\infty}(\mathbb{R})$ is an even function. Then,
\begin{equation}\label{eq:selberg_trace_formula}
\sum_j \hat{\phi}\left( \sqrt{\lambda_j - \frac{1}{4}}\right) = \frac{\Vol (X)}{2 \pi} \int_{0}^\infty \hat{\phi} (r) r\tanh\left(\pi r\right) dr+ \sum_{\gamma \in \calP(X)} \sum_{m=1}^{\infty} \frac{\ell_{\gamma}(X)}{2\sinh\left(\frac{ m\ell_{\gamma}(X)}{2}\right) } \phi(m\ell_{\gamma}(X)).
\end{equation}   
\end{lem}

We quote the  following elaboration of the Markov brothers' inequality from \cite[Lemma 2.1]{MPvH25}. 
\begin{lem}[Markov brothers' inequality]\label{lem:markov_brothers}
For every real polynomial $P$ of degree at most $q$ and every $k \in \bbN$,
\begin{equation}\label{eq:markov}
\sup_{t \in [0, \frac{1}{2q^2}]} |P^{(k)} (t)| \leq \frac{2^{2k + 1} q^{4k}}{(2k-1)!!} \sup_{n \geq q^2} \left|P\left(\frac{1}{n}\right)\right|,
\end{equation}
where $(2k-1)!! \coloneqq (2k-1)(2k-3)\cdots 3\cdot 1$.
\end{lem}

We recall the Margulis lemma on $\mathbb{H}^2$ (\cite[\S 12.6]{Ratcliffe}):
\begin{lem}[Margulis lemma on $\bbH^2$]\label{lem:margulis}
    There exists a constant $\epsilon_0>0$ such that for any hyperbolic surface $X$ and $z\in X$, there is at most one primitive geodesic loop (up to inverse) based at $z$ with length $\leq\epsilon_0$.
\end{lem}

\begin{notation}\label{notation:margulis}
We call $\epsilon_0$ the \emph{Margulis constant} for $\bbH^2$. Any primitive geodesic loop with length less than or equal to $\epsilon_0$ is a \emph{short geodesic loop}. The Margulis constant of $\mathbb{H}^2$ is $\epsilon_0=2\mathrm{arcsinh}(1)$, see \cite[Theorem 2]{Yam81}.
\end{notation}

We now prove the following lemma on the geometry of hyperbolic surfaces. 

\begin{lem}\label{lem:collar-2}
    For $0<\delta<\frac{1}{2}\mathrm{arcsinh}(1)$, we have
    \begin{equation*}
        d(X_{2\delta},\partial X_{\delta})\geq \log 2.
    \end{equation*}
\end{lem}
\begin{proof}
    Suppose that $z\in X$ has injectivity radius $\leq 2\delta$. By \cite[Theorem 4.1.6]{Bu10}, we know $z$ is in the collar of a short geodesic $\beta$ with length $\ell\leq 2\mathrm{arcsinh}(1)$. Suppose the collar has boundary $\partial \mathcal{C}$, then again using \cite[Theorem 4.1.6]{Bu10}, we have
    \begin{equation*}
        \sinh (\injrad_z(X)) =  \cosh \frac{\ell}{2}\cosh d(z,\partial \mathcal{C})-\sinh d(z,\partial\mathcal{C})=\sinh \frac{\ell}{2}\cosh d(z,\beta).
    \end{equation*}
    Suppose $z_1$ has injectivity radius $\delta$ and $z_2$ is the closest point with injectivity radius $2\delta$. Note that $z_1$ and $z_2$ must lie in the same collar. Then,
    \begin{equation*}
        d(z_2,\beta)-d(z_1,\beta) = \mathrm{arccosh}\left(\frac{\sinh 2\delta}{\sinh\frac{\ell}{2}}\right)-\mathrm{arccosh}\left(\frac{\sinh \delta}{\sinh\frac{\ell}{2}}\right) \geq \log (2\cosh \delta)\geq \log 2.
    \end{equation*}
\end{proof}

Finally, we have the following bound on the number of geodesic loops of bounded length.
\begin{lem}
\label{lem:geodesic_loop_bound}
Let $Y = \Gamma \backslash \bbH$ be a hyperbolic surface and $\delta_0 \in (0,1)$. Then for $z \in Y_{\delta_0}$,
$$\#\{\gamma \in \Gamma : d_\bbH(z, \gamma z) \leq L\} \lesssim \frac{e^{L}}{\delta_0},$$
where the implicit constant holds uniformly over the choice of $Y$, $\delta_0$ and $z \in Y_{\delta_0}$.
\end{lem}

\begin{proof}
 We lift $z \in Y_{\delta_0}$ to $\bbH$. Suppose $z$ satisfies $\injrad_z(Y)\geq \mathrm{arcsinh}(1)/2$. Then the balls  $B_{\gamma z}(\mathrm{arcsinh}(1)/2)$ (indexed by $\gamma \in \Gamma$) are pairwise disjoint. If $d_{\bbH} (z, \gamma z) \leq L$, then $B_{\gamma z}(\mathrm{arcsinh}(1)/2) \subset B_z(L + \mathrm{arcsinh}(1)/2)$. Thus, there are at most $C e^L $ such  $\gamma$ for a uniform constant $C>0$. On the other hand, if $\injrad_z(Y) < \mathrm{arcsinh}(1)/2$, then by \cite[Theorem 4.1.6]{Bu10}, $z$ is contained in the collar of a short geodesic $\beta$ with length $\ell \leq 2\mathrm{arcsinh}(1)$. Suppose the collar of $\beta$ is given by $\{(\rho,\theta): |\rho|\leq \rho_1, \theta \in [0,1)\}$ ($\rho_1\geq \mathrm{arcsinh}(1)$) with the metric $d\rho^2+\ell^2\cosh^2(\rho)d\theta^2$ and $z$ has coordinate $\rho=\rho_0>0$. Take the neighborhood $\mathcal{B}=\{\rho_0-\mathrm{arcsinh}(1)\leq\rho\leq \rho_0\}$. Then the sets $\gamma \mathcal{B}$ (indexed by $\gamma \in \Gamma$) are pairwise disjoint. Moreover, $\mathcal{B}$ has area $\gtrsim \delta_0$. If $d_{\bbH} (z, \gamma z) \leq L$, then $\gamma \mathcal{B} \subset B_z(L + C)$. Thus, there are at most $C e^L /\delta_0$ such  $\gamma$ for a uniform constant $C>0$.
\end{proof}

\subsection{Test function}\label{subsection:test}

For our application of the Selberg pre-trace formula~\eqref{eq:pretrace-1} and the trace formula~\eqref{eq:selberg_trace_formula},  we want to use a $\hat{\phi}$ that approximates an interval of eigenvalues and is suited to  the polynomial method. We now construct such a $\hat{\phi}$.

We begin with the function $f:\bbR\cup i\bbR\to \bbR$ from \cite[\S 2]{HMT25b}. Specifically, $f$ has the following properties:
\begin{enumerate}
    \item $f$ is smooth.
    \item $f$  is non-negative on $\bbR\cup i\bbR$. 
    \item $\check{f}$ is smooth, even, non-negative, and supported in $[-1,1]$. Since
    \begin{equation}\label{eq:f'(it)}
        \frac{d}{dt} f(it)= \int   \xi e^{t \xi}\check{f}(\xi) d\xi >0,\quad t>0,
    \end{equation}
    the map $t \in [0, \infty) \mapsto f(ti)$ is strictly increasing. 
    Moreover, we have
    $f(0)>0$, and $ f([0, \infty)) \subset [0, f(0)]$. 
    \item
From the smoothness of $\check{f}$, we know
\begin{equation}\label{eq:decay-f}
    |f(x)|\leq C_N(1+|x|)^{-N}, \quad
     x\in \bbR
\end{equation}
for any $N\in \bbN$. 
    \item Since
\begin{equation}\label{eq:f''(0)}
    f'(0)=\int_{\bbR} -i\xi\check{f}(\xi) d\xi = 0,\quad f''(0)=\int_{\bbR} -\xi^2\check{f}(\xi) d\xi<0,
\end{equation}
there exists $c_0 \in (0,1/2)$ such that
\begin{equation}\label{eq:f-monotone}
\begin{split}
   -C x\leq f'(x)\leq -c_0 x,\quad 0<f(4c_0)\leq f(x)&\leq f(0)-\frac{c_0}{2} x^2,\quad x\in [0,4c_0],\\
    f(x)&\leq f(4c_0),\qquad \quad\, x>4c_0.
\end{split}
\end{equation}
\end{enumerate}

We will use a rescaled version of $f$:
\begin{equation*}\label{eq:def-fLa}
    f_{\Lambda_0}(x) \coloneqq f(c_0\Lambda_0^{-1/2}x),\quad \Lambda_0\in [1/4,\infty).
\end{equation*}
Note that $f_{\Lambda_0}(\sqrt{\Delta_{X_n} - \tfrac{1}{4}})$ is a bounded operator. The spectrum of $f_{\Lambda_0}(\sqrt{\Delta_{X_n} - \tfrac{1}{4}})$ is contained in $[0, f_{\Lambda_0}(\tfrac{i}{2})]$, where $[f_{\Lambda_0}(0), f_{\Lambda_0}(\tfrac{i}{2})]$ corresponds to the eigenvalues of $\Delta_{X_n}$ below $\tfrac{1}{4}$ and $f_{\Lambda_0}([0, \infty))$ corresponds to the eigenvalues of $\Delta_{X_n}$ above $\tfrac{1}{4}$.

Let $h(x)=x\tilde{h}(x)$ be a polynomial of degree $q\geq 1$. We use the notation
\begin{equation*}\label{eq:norm}
    \|\tilde{h}\|:=\sup_{x\in [0,f(i/2)]} |\tilde{h}(x)|,\quad \|\tilde{h}\|_{C^\kappa}=\sum_{0\leq j\leq \kappa}\|\partial_x^j\tilde{h}\|.
\end{equation*} 
Note that since $c_0 \Lambda_0^{-1/2} \leq 1$, $f_{\Lambda_0}([0, \infty)) \cup f_{\Lambda_0} ([0, i/2]) \subset [0,f(i/2)]$.

\begin{lem}\label{lem:unif-K}
    Let $f_{\Lambda_0}$ be defined as above. Let $h(x)=x\tilde{h}(x)$ where $\tilde{h}(x)$ is a continuous function on $[0,f(i/2)]$, and $t_\gamma \coloneqq d_{\bbH}(z, \gamma z)$ for $\gamma \in \Gamma\setminus \{\id\}$. Then for
    \begin{equation*}
        K_{(h \circ f_{\Lambda_0})^{\vee}} (z, \gamma z) = k_{(h \circ f_{\Lambda_0})^{\vee}} (t_\gamma)=-\frac{1}{\sqrt{2}\pi}\int_t^{\infty}\frac{((h\circ f_{\Lambda_0})^{\vee})'(s)}{\sqrt{\cosh s -\cosh t_\gamma}} ds
    \end{equation*}
    and $c_0\Lambda_0^{-1/2}<1$, we have
    \begin{equation*}
        |k(t_\gamma)|\lesssim \left\{\begin{array}{ll}
         \int_{\mathbb{R}} |h\circ f_{\Lambda_0}(r)| r^2 dr \lesssim c_0^{-3}\Lambda_0^{3/2} e^{-t_\gamma/2}\|\tilde{h}\|,     & 0<t_\gamma<1, \\
         e^{-t_\gamma/2}\int_{\mathbb{R}} |h\circ f_{\Lambda_0}(r)| |r| dr \lesssim c_0^{-2}\Lambda_0 e^{-t_\gamma/2}\|\tilde{h}\|,     & t_\gamma \geq 1.
        \end{array}\right. 
    \end{equation*}

\end{lem}
\begin{proof}
For simplicity of notation, set $t = t_\gamma$. Recall that 
    \begin{equation*}
        \cosh s - \cosh t =2\sinh\left(\frac{s+t}{2}\right) \sinh\left(\frac{s-t}{2}\right).
    \end{equation*}
If $t\geq 1$, this implies that
\begin{equation*}
    |k(t)|\lesssim \left(e^{-t/2} \int_t^{t+1} \frac{1}{\sqrt{s-t}} ds +\int_{t+1}^{\infty} e^{-s/2}  ds \right)\int_{\mathbb{R}} |h\circ f_{\Lambda_0}(r)| |r|dr \lesssim c_0^{-2}\Lambda_0 e^{-t/2}\|\tilde{h}\|.
\end{equation*}
    
If $t<1$, then we have
    \begin{equation*}
        |k(t)|\lesssim \int_t^{1} \frac{|((h\circ f_{\Lambda_0})^{\vee})'(s)|}{\sqrt{s^2-t^2}}ds +c_0^{-2}\Lambda_0 \|\tilde{h}\|.
    \end{equation*}
    Since $f_{\Lambda_0}$ is even, $((h\circ f_{\Lambda_0})^{\vee})'(0)=0$. Thus we have
    \begin{equation*}
    \begin{split}
        |k(t)| &\lesssim \int_t^{1} \frac{s}{\sqrt{s^2-t^2}}ds \left(\sup_{r} |((h\circ f_{\Lambda_0})^{\vee})''(r)|\right) +c_0^{-2}\Lambda_0 \|\tilde{h}\|\\
        &\lesssim \int_t^{1} \frac{s}{\sqrt{s^2-t^2}}ds \int_{\mathbb{R}} |h\circ f_{\Lambda_0}(r)| r^2 dr +c_0^{-2}\Lambda_0 \|\tilde{h}\|\lesssim c_0^{-3}\Lambda_0^{3/2} \|\tilde{h}\|.\qedhere
    \end{split}
    \end{equation*}
\end{proof}

\section{Proof of Theorem~\ref{thm:eigen-g}}\label{section:proof-theorem3}
In this section, we prove Theorem~\ref{thm:eigen-g} following the outline in \S\ref{subsection:proof_idea}. Let $Y$ be a random family of compact oriented hyperbolic surfaces of area $2\pi(2g-2)$.  

Recall the definition of $f_{\Lambda_0}$ from \S\ref{subsection:test}. When we apply \eqref{eq:pretrace-1}, we will set $\hat{\phi}(x) = h\circ f_{\Lambda_0}(x)$, where $h(x)$ is a polynomial of degree $q$ such that $h(x)=x\tilde{h}(x)$.

\subsection{Proof of Theorem~\ref{thm:eigen-g}}\label{subsec:eigen-g}
Before the proof, we introduce some notations. For simplicity, we will deal with the case when $Y$ is connected (otherwise, just consider each connected component of $Y$). Let $Y=\Gamma\backslash\mathbb{H}$. 
Let $\Gamma_0$ be the set of primitive elements in $\Gamma\setminus \{\id\}$. For $z\in \bbH $, we define
\begin{equation}\label{eq:loc-osci}
    S(h\circ f_{\Lambda_0})(z) := \frac{1}{2}\sum_{k_1,k_2,k_3,k_4\in \bbZ\setminus \{0\}}\sum_{\gamma\in \Gamma_0}\prod_{\ell=1}^{4}K_{(h\circ f_{\Lambda_0})^{\vee}}(z,\gamma^{k_\ell}z)
\end{equation}
and
\begin{equation}\label{eq:local-weyl-Vn}
    V(h\circ f_{\Lambda_0})(z):=\left(\sum_{j} h\circ f_{\Lambda_0}(\sqrt{\lambda_j-1/4}) |u_j(z)|^2 - \frac{1}{2\pi }\int_0^{\infty} h\circ f_{\Lambda_0}(r) r\tanh \pi r dr\right)^4.
\end{equation}
Using the pre-trace formula \eqref{eq:pretrace-1}, we see that $S(h \circ f_{\Lambda_0})$ is the diagonal component of $V(h\circ f_{\Lambda_0})$. More specifically, $V(h\circ f_{\Lambda_0})$ can be written as a sum over $k_1, \ldots, k_4 \in \bbZ \setminus\{0\}$ and $\gamma_1, \ldots, \gamma_4 \in \Gamma_0$, while $S(h \circ f_{\Lambda_0})$ is a diagonal sum over $k_1, \ldots, k_4 \in \bbZ \setminus\{0\}$ and $\gamma \in \Gamma_0$. The factor $1/2$ in \eqref{eq:loc-osci} is due to double counting since $\gamma^k =(\gamma^{-1})^{-k}$.

We will prove Theorem~\ref{thm:eigen-g} assuming the following proposition, a bound on the off-diagonal component. This proposition is proved via the polynomial method in \S\ref{subsec:proof_31}. In the following, we will always assume 
\begin{equation}\label{eq:c0-assumption}
\Lambda_0\geq 1/4, \quad 2c_0< \epsilon_0,
\end{equation}
where $\epsilon_0$ is Margulis constant in Lemma~\ref{lem:margulis}. 

\begin{prop}\label{prop:expect-fn}
Under the assumptions of Theorem~\ref{thm:eigen-g}, for $f_{\Lambda_0}(x)$ from \S\ref{subsection:test} and a degree $q$ polynomial $h(x)=x \tilde{h}(x)$,
there exists $C=C(\kappa, C_0)>0$ such that
    \begin{equation}\label{eq:expect-fn-prop}
        \bbE\left[\int_{Y_{\delta_0}} \left(V(h\circ f_{\Lambda_0})(z)-S(h\circ f_{\Lambda_0})(z)\right)^2 dz \right]\leq  C \delta_0^{-8}\frac{\Lambda_0^{12} q^{4\kappa}}{ n}\|\tilde{h}\|^8.
    \end{equation}
    Moreover, if one replaces $V(h\circ f_{\Lambda_0})(z)$ and $S(h\circ f_{\Lambda_0})(z)$ with  multilinear expressions in four different polynomials $h_1,h_2,h_3,h_4$ with degrees $q_1, q_2, q_3, q_4\leq q$:
    \begin{equation}\label{eq:generalize-V}
    \begin{split}
    &V(h_1\circ f_{\Lambda_0},h_2\circ f_{\Lambda_0},h_3\circ f_{\Lambda_0},h_4\circ f_{\Lambda_0})(z)\\
    &:=\prod_{m=1}^4\left(\sum_{j} h_m\circ f_{\Lambda_0}(\sqrt{\lambda_j-1/4}) |u_j(z)|^2 - \frac{1}{2\pi }\int_0^{\infty} h_m\circ f_{\Lambda_0}(r) r\tanh \pi r dr\right),
    \end{split}
\end{equation}
and
\begin{equation}\label{eq:S-def}
S(h_1\circ f_{\Lambda_0},h_2\circ f_{\Lambda_0},h_3\circ f_{\Lambda_0},h_4\circ f_{\Lambda_0})(z) := \frac{1}{2}\sum_{k_1,k_2,k_3,k_4\in \bbZ\setminus \{0\}}\sum_{\gamma\in \Gamma_0}\prod_{\ell=1}^{4}K_{(h_\ell \circ f_{\Lambda_0})^{\vee}}(z,\gamma^{k_\ell}z),
\end{equation}
we get an analogous estimate to~\eqref{eq:expect-fn-prop} with error $C \delta_0^{-8} \frac{\Lambda_0^{12} q^{4\kappa}}{n} \|\tilde{h}_1\|^2  \|\tilde{h}_2\|^2 \|\tilde{h}_3\|^2 \|\tilde{h}_4\|^2$.
\end{prop}

We can rewrite \eqref{eq:expect-fn-prop} in the following way. Importantly, we generalize the condition that $h(x)=x \tilde{h}(x)$  is a polynomial of degree $q$; it now can be any smooth function. If we replace $V(h\circ f_{\Lambda_0})$ and $S(h \circ f_{\Lambda_0})$ with something of the form \eqref{eq:generalize-V} and \eqref{eq:S-def}, respectively, the corresponding inequality holds. 
\begin{cor} \label{cor:expect-fn-prop}
For any smooth function $h(x)=x\tilde{h}(x)$, $\tilde{h}(x)\in C^{\infty}([0,f(i/2)])$, we have
    \begin{equation}\label{eq:expect-fn-general}
    \mathbb{E}\left[\int_{Y_{\delta_0}}  \left(V(h\circ f_{\Lambda_0})(z)-S(h\circ f_{\Lambda_0})(z)\right)^2 dz \right] \leq  \frac{C\Lambda_0^{12}}{ n\delta_0^8}\|\tilde{h}\|_{C^{2\kappa+1}}^8.
\end{equation}

\end{cor}
\begin{proof}

Let $\tilde{h}(x)$ be a polynomial of degree $q-1$. We write $\tilde{h}(x)=\sum_{j=0}^{q-1} a_j T_j(2f(i/2)^{-1}x-1)$, where $T_j$ is the $j$-th Chebyshev polynomial as in \cite[(4.1)]{CGVTvH24}. Denoting $\tilde{T}_j(x) \coloneqq xT_j(2f(i/2)^{-1}x-1)$, we have
\begin{equation*}
    \begin{split}
        &\bbE\left[\int_{Y_{\delta_0}} \left(V(h\circ f_{\Lambda_0})(z)-S(h\circ f_{\Lambda_0})(z)\right)^2 dz \right]\\
        &\leq \left(\sum_{j_1,\ldots,j_4} \left(\bbE\left[\int_{Y_{\delta_0}} \left(V(a_{j_1} \tilde{T}_{j_1}\circ f_{\Lambda_0},\ldots,a_{j_4} \tilde{T}_{j_4}\circ f_{\Lambda_0})(z)-S(a_{j_1} \tilde{T}_{j_1}\circ f_{\Lambda_0},\ldots,a_{j_4} \tilde{T}_{j_4}\circ f_{\Lambda_0})(z)\right)^2 dz \right]\right)^{1/2}\right)^2\\
        &\leq \frac{C\Lambda_0^{12} }{ n\delta_0^8}\left(\sum_{j=0}^{q-1}  (j+1)^{2\kappa}|a_j|\right)^8\\
        &\leq \frac{C \Lambda_0^{12}}{ n\delta_0^8}\|\tilde{h}\|_{C^{2\kappa+1}}^8,
    \end{split}
\end{equation*}
where the first step is due to the Minkowski inequality, the second step is due to \eqref{eq:expect-fn-prop}, and the last step is due to \cite[Corollary 4.5]{CGVTvH24}. By approximating smooth functions by polynomials, we note that \eqref{eq:expect-fn-general} holds for general smooth functions $h(x)=x \tilde{h}(x)$, $\tilde{h}(x) \in C^{\infty}([0,f(i/2)])$.
\end{proof}

Using Corollary~\ref{cor:expect-fn-prop}, we now prove Theorem~\ref{thm:eigen-g}. 

\begin{proof}[Proof of Theorem~\ref{thm:eigen-g}]
1.
Let $\Lambda\geq 0$ and $\alpha_0=\frac{1}{16(2\kappa+1)}$. Although Theorem~\ref{thm:eigen-g} is stated for $\Lambda\geq 1/4$, it is convenient to establish the auxiliary spectral window estimate for all $\Lambda\geq 0$, in order to include eigenvalues below $1/4$. Let $\Lambda_0=\max(\Lambda,1/4)$ for $\Lambda\geq 0$.

By \cite[Lemma 3.3]{DZ16}, we take $h_{\Lambda}(x)\in C^{\infty}([0,f(i/2)])$ so that
\begin{equation*}\label{eq:hLambda}
    h_{\Lambda}(x) =\left\{\begin{array}{ll}
    1,  &   x\in [f_{\Lambda_0}(\sqrt{\Lambda+(1+\Lambda)n^{-\alpha_0}-1/4}),f_{\Lambda_0}(\sqrt{\Lambda-1/4})], \\
    0,     &   x\in [0,f_{\Lambda_0}(\sqrt{\Lambda+2(1+\Lambda) n^{-\alpha_0}-1/4})],\\
    0, &x\in [f_{\Lambda_0}(\sqrt{\Lambda-(1+\Lambda) n^{-\alpha_0}-1/4}),f(i/2)].
    \end{array}\right.
\end{equation*}
and moreover,
\begin{equation}\label{eq:h-properties}
    0\leq h_{\Lambda}(x)\leq 1,\quad h_{\Lambda}(x)=x\tilde{h}_{\Lambda}(x),\quad |\tilde{h}_{\Lambda}^{(j)}(x)|\leq C_j n^{j\alpha_0},\quad x\in [0,f(i/2)],
\end{equation}
where the last estimate is due to the following calculation for $k = \pm1$ using \eqref{eq:f'(it)}, \eqref{eq:f-monotone} and the mean value theorem when $\Lambda$ is away from $1/4$:
\begin{equation}\label{eq:MVT}
\begin{split}
&f_{\Lambda_0}(\sqrt{\Lambda +k(1+\Lambda) n^{-\alpha_0} -1/4})-f_{\Lambda_0}(\sqrt{\Lambda +(k+1)(1+\Lambda) n^{-\alpha_0} -1/4})\\
&\geq C^{-1}\Lambda_0^{-1/2}\left|\sqrt{\Lambda + k(1+\Lambda) n^{-\alpha_0}-1/4}\right|\\
&\quad \cdot \left|\Lambda_0^{-1/2}\sqrt{\Lambda +(k+1)(1+\Lambda) n^{-\alpha_0} -1/4}-\Lambda_0^{-1/2} \sqrt{\Lambda +k(1+\Lambda) n^{-\alpha_0} -1/4}  \right|\\
& \geq C^{-1}n^{-\alpha_0}.
\end{split}
\end{equation}
When $\Lambda$ is in a small neighbourhood of $1/4$, by \eqref{eq:f''(0)} we have
\begin{equation*}
    \begin{split}
        &f_{\Lambda_0}(\sqrt{\Lambda +k(1+\Lambda) n^{-\alpha_0} -1/4})-f_{\Lambda_0}(\sqrt{\Lambda +(k+1)(1+\Lambda) n^{-\alpha_0} -1/4})\\
        &\geq -C^{-1}\left(\left(\sqrt{\Lambda +k(1+\Lambda) n^{-\alpha_0} -1/4}\right)^2- \left(\sqrt{\Lambda +(k+1)(1+\Lambda) n^{-\alpha_0} -1/4}\right)^{2}\right)\\
        &\geq C^{-1}n^{-\alpha_0}.
\end{split}
\end{equation*}

By a similar argument, we can conclude
\begin{equation}\label{eq:supp-bound}
\left| \supp (h_{\Lambda}\circ f_{\Lambda_0})\right| \lesssim \sqrt{\Lambda_0} n^{-\alpha_0},\quad \Lambda\geq 1.
\end{equation}

From Corollary~\ref{cor:expect-fn-prop} and~\eqref{eq:h-properties}, we have
\begin{align*}
&\bbE \left[ \int_{Y_{\delta_0}}  (V(h_{\Lambda}\circ f_{\Lambda_0})(z)-S(h_{\Lambda}\circ f_{\Lambda_0})(z))^2 dz \right]\\
&\quad \leq C \frac{\Lambda_0^{12}}{n\delta_0^8} \| \tilde{h}_{\Lambda}\|^8_{C^{2\kappa +1}} \leq C \frac{\Lambda_0^{12}}{n\delta_0^8} n^{8 (2\kappa +1) \alpha_0} = C \frac{\Lambda_0^{12}}{\delta_0^{8} n^{1/2}}.
\end{align*}

Let $C_1>0$ be a large constant to be determined later. Then by Chebyshev's inequality,  there is a subset $\Omega_g(\Lambda)$ of the family of random hyperbolic surfaces of genus $g$ with
\begin{equation}\label{eq:prob-g}
    \mathbb{P}(\Omega_g(\Lambda))\geq 1-(C_1\Lambda_0)^{-2}n^{-1/4},
\end{equation}
such that for $Y\in \Omega_g(\Lambda)$ we have
\begin{equation}\label{eq:eigenfn_cheby}
    \int_{Y_{\delta_0}}  (V(h_{\Lambda}\circ f_{\Lambda_0})(z)-S(h_{\Lambda}\circ f_{\Lambda_0})(z))^2 dz\lesssim \Lambda_0^{14} \delta_0^{-8}n^{-1/4}.
\end{equation}

2. We now find a set $\Omega_g$ that does not depend on $\Lambda$. Take $\Lambda(j)= \Lambda_0(j) := 2^{j-2}$ for $j\geq 1$, and $\Lambda_0(0)=1/4$, $\Lambda(0)=0$. Set
\begin{equation*}
    \Lambda(j,\ell) := \Lambda(j)+\ell n^{-\alpha_0}\Lambda_0(j)\in [\Lambda(j),\Lambda(j+1)],\,\,  0\leq\ell \leq L_{j}:= \left\lfloor n^{\alpha_0} \frac{\Lambda(j+1)-\Lambda(j)}{\Lambda_0(j)}\right\rfloor ,
\end{equation*}
and
\begin{equation*}
    \Omega_g:=\bigcap_{j=0}^{\infty}\bigcap_{\ell=0}^{L_j}\Omega_g(\Lambda(j,\ell)).
\end{equation*}
Taking $C_1$ in \eqref{eq:prob-g} sufficiently large, we have
\begin{equation*}
\begin{split}
    \mathbb{P}(\Omega_g)&\geq 1-\sum_{j=0}^{\infty}\sum_{\ell=0}^{L_j} (C_1\Lambda_0(j,
    \ell))^{-2}n^{-1/4}\\
    &\geq 1-\sum_{j=0}^{\infty} (C_1\Lambda_0(j))^{-2}C n^{\alpha_0}n^{-1/4}\\
    &\geq 1-n^{-1/10}. 
\end{split}
\end{equation*}
Now \eqref{eq:eigenfn_cheby} holds for all $\Lambda=\Lambda(j,\ell)$, $\Lambda_0=\Lambda_0(j)$, and $Y \in \Omega_g$.

3. Let $z_0 \in Y$. Using~\eqref{eq:eigenfn_cheby} and the Selberg pretrace formula~\eqref{eq:pretrace-1},
we have 
\begin{equation}\label{eq:local-weyl-i}
\begin{split}
   &\sum_{\lambda_j(Y)\in [\Lambda(i,\ell),\Lambda(i,\ell)+(1+\Lambda(i,\ell))n^{-\alpha_0}]}\int_{B(z_0,\log 2)\cap Y_{\delta_0}}|u_j(z)|^2 dz\\
   &\quad \lesssim \left(\int_{B(z_0,\log 2)\cap Y_{\delta_0}} \left(\sum_{\lambda_j(Y)\in [\Lambda(i,\ell),\Lambda(i,\ell)+(1+\Lambda(i,\ell))n^{-\alpha_0}]}|u_j(z)|^2\right)^8 dz\right)^{1/8}\\
   &\quad \lesssim \int_0^{\infty}h_{\Lambda(\ell)}\circ f_{\Lambda_0(\ell)}(r) r\tanh \pi r dr +\sup_{z\in Y_{\delta_0}}|S(h_{\Lambda(\ell)}\circ f_{\Lambda_0(\ell)})(z)|^{1/4} + \Lambda_0^{7/4}\delta_0^{-1} n^{-1/32}.
\end{split}
\end{equation}
From~\eqref{eq:supp-bound} and a direct calculation using \eqref{eq:f''(0)} when $\Lambda\leq 1$, we have
\begin{equation}\label{eq:int-bound}
    \int_0^{\infty}h_{\Lambda}\circ f_{\Lambda_0}(r) r dr  \lesssim \Lambda_0 n^{-\alpha_0},\quad  \int_0^{\infty}h_{\Lambda}\circ f_{\Lambda_0}(r) r^2 dr  \lesssim \Lambda_0^{3/2} n^{-\alpha_0}.
\end{equation}
From Lemma~\ref{lem:margulis}, there is at most one short geodesic loop based at $z$. We assume such a geodesic loop exists and call it $\gamma_0(z)$.  Using Lemma~\ref{lem:unif-K}, Lemma~\ref{lem:geodesic_loop_bound},~\eqref{eq:int-bound}, and the fact that $z \in Y_{\delta_0}$,
\begin{equation}\label{eq:S-bound}
\begin{split}
    |S(h_{\Lambda}\circ f_{\Lambda_0})(z)| &\lesssim \sum_{\gamma\in \Gamma_0}\left(\sum_{k\in \bbZ\setminus \{0\}}|K_{(h_{\Lambda}\circ f_{\Lambda_0})^{\vee}}(z,\gamma^{k}z)|\right)^4 \\
     &\lesssim \left(\int_{\bbR} h_{\Lambda}\circ f_{\Lambda_0}(r) |r| dr\right)^4 \sum_{\gamma_0(z)\neq \gamma\in \Gamma_0}\left(\sum_{k \in \bbZ\setminus\{0\}} \exp\left(-d_{\bbH}(z,\gamma^{k}z)/2\right)\right)^4\\
    &+ \left(\int_{\bbR} h_{\Lambda}\circ f_{\Lambda_0}(r) (|r|+r^2) dr\right)^4 \left(\sum_{k\in \mathbb{Z}\setminus\{0\}}\exp\left({-d_{\mathbb{H}}(z,\gamma_0(z)^k z)}/2\right)\right)^4\\
    &\lesssim \Lambda_0^6 \delta_0^{-4} n^{-4\alpha_0}.
\end{split}
\end{equation}

We then conclude from~\eqref{eq:local-weyl-i},~\eqref{eq:int-bound}, and~\eqref{eq:S-bound} for any $\lambda_j(Y)\leq \Lambda_0$,
\begin{equation*}
    \int_{B(z_0,\log 2)\cap Y_{\delta_0}}|u_j(z)|^2 dz\lesssim \Lambda_0^{7/4} \delta_0^{-1} n^{-\alpha_0}.
\end{equation*}

4.
Finally, we use the Sobolev embedding (\cite[\S 5.6]{evans}) and elliptic estimate (\cite[\S 6.3]{evans}). For $z_0\in Y_{2\delta_0}$, by Lemma~\ref{lem:collar-2}, we have $B(z_0,\log 2)\subset Y_{\delta_0}$. So
\begin{equation*}
\begin{split}
    \|u_j\|_{L^{\infty}(B(z_0,\frac{1}{4}\log 2))}^2 &\lesssim  \|u_j\|_{H^2(B(z_0,\frac{1}{2}\log 2))}^2\\
    &\lesssim \|u_j\|_{L^2(B(z_0,\log 2))}^2+ \|\Delta_{\bbH}u_j\|_{L^2(B(z_0,\log 2))}^2\\
    &\lesssim \Lambda_0^{2}\int_{B(z_0,\log 2)}|u_j(z)|^2 dz\lesssim \delta_0^{-1} \Lambda_0^4 n^{-\alpha_0}.
\end{split}
\end{equation*}
Since $z_0$ is arbitrary, we conclude 
\begin{equation*}
    \|u_j\|_{L^\infty(Y_{2\delta_0})}\leq C\Lambda^{2}\delta_0^{-1/2}n^{-\alpha_0/2}\|u_j\|_{L^{2}(Y)}.\qedhere
\end{equation*}
\end{proof}

\subsection{Proof of Proposition~\ref{prop:expect-fn}}\label{subsec:proof_31}

Although in Corollary~\ref{cor:expect-fn-prop},  we generalized Proposition~\ref{prop:expect-fn} to hold for smooth functions, in this section, we return to assuming that $h(x) = x \tilde{h}(x)$ is a degree $q$ polynomial. We use some of the proof methods developed in \cite{MPvH25}.

For simplicity of notation, we only discuss the case of a single polynomial $h(x)=x\tilde{h}(x)$, i.e., \eqref{eq:expect-fn-prop}. The case with four different polynomials $h_\ell(x)=x\tilde{h}_{\ell}(x)$, $\ell=1,2,3,4$ (i.e., \eqref{eq:generalize-V}) follows from the same proof.

We first prove an initial bound.

\begin{lem}\label{lem:unif-fn}
Suppose that $c_0>0$ is sufficiently small, and $\Lambda_0\geq 1/4$.
We have the bound
\begin{equation}\label{eq:expect-fn-sq-lem}
\begin{split}
    &\bbE\left[\frac{1}{\Vol(Y)}\int_{Y_{\delta_0}} \left(V(h\circ f_{\Lambda_0})(z)-S(h\circ f_{\Lambda_0})(z)\right)^2 dz\right] \lesssim \Lambda_0^{12}\delta_0^{-8}\|\tilde{h}\|^8.
\end{split}
\end{equation}
\end{lem}
\begin{proof}
1. In this first step, we show for $c_0>0$ sufficiently small and $\Lambda_0\geq 1/4$, we have
    \begin{equation}\label{eq:unif-fn}
        \sup_{z\in Y_{\delta_0}}\left|\sum_{j}( h\circ f_{\Lambda_0})(\sqrt{\lambda_j-1/4}) |u_j(z)|^2\right| \lesssim  \Lambda_0^{3/2}  \delta_0^{-1}\|\tilde{h}\|.
    \end{equation}
Since $h(x)=x\tilde{h}(x)$, we have $h\circ f_{\Lambda_0}= f_{\Lambda_0}\cdot (\tilde{h}\circ f_{\Lambda_0}) $. Therefore,
    \begin{equation*}
        \left|\sum_{j} (h\circ f_{\Lambda_0})(\sqrt{\lambda_j-1/4}) |u_j(z)|^2 \right|\leq \|\tilde{h}\| \sum_{j}  f_{\Lambda_0}(\sqrt{\lambda_j-1/4}) |u_j(z)|^2.
    \end{equation*}
    By the Selberg pre-trace formula \eqref{eq:pretrace-1} (using Notation~\ref{notation:K}),
    \begin{equation}\label{eq:upper-pre-trace}
            \sum_{j}  f_{\Lambda_0}(\sqrt{\lambda_j-1/4}) |u_j(z)|^2 \leq \frac{1}{2\pi}\int_0^{\infty}f_{\Lambda_0}(r) r\tanh \pi r dr+\sum_{\gamma \in \Gamma\setminus \{\id\}}  |K_{f_{\Lambda_0}^{\vee}}(z,  \gamma z)|.
    \end{equation}

    We begin with the first term on the right-hand side of~\eqref{eq:upper-pre-trace}. By the rapid decay of $f$ in \eqref{eq:decay-f},
    \begin{equation}\label{eq:unif-bound-1}
        \int_{0}^\infty f_{\Lambda_0} (r) r \tanh\left(\pi r\right) dr \lesssim \Lambda_0.
    \end{equation}

    For the second term of the right-hand side of \eqref{eq:upper-pre-trace}, we note that by \eqref{eq:selberg-trans} and \begin{equation}\label{eq:support}
    |\check{f}_{\Lambda_0}|\lesssim \Lambda_0^{1/2},\quad \supp \check{f}_{\Lambda_0} \subset [-c_0\Lambda_0^{-1/2},c_0\Lambda_0^{-1/2}],
   \end{equation}
    we have
    \begin{equation*}\label{eq:support-fn}
        \supp k_{f^\vee_{\Lambda_0} }\subset [0,c_0\Lambda_0^{-1/2}].
    \end{equation*}
Therefore, by ~\eqref{eq:c0-assumption} and Lemma~\ref{lem:margulis},    $\sum_{\gamma \in \Gamma\setminus \{\id\}}  |K_{f_{\Lambda_0}^{\vee}}(z,  \gamma z)|$  is either zero or a sum over $\gamma_0^k$, where $\gamma_0$ is a short geodesic loop and for $k\in \mathbb{Z}\setminus\{0\}$.
Assuming $\sum_{\gamma \in \Gamma\setminus \{\id\}}  |K_{f_{\Lambda_0}^{\vee}}(z,  \gamma z)|$ is nonzero, for $z\in Y_{\delta_0}$,
    \begin{equation}\label{eq:K-bound}
    \begin{split}
        \sum_{\gamma \in \Gamma\setminus \{\id\}}  |K_{f_{\Lambda_0}^{\vee}}(z,  \gamma z)| = \sum_{k\neq 0}|K_{f_{\Lambda_0}^{\vee}}(z,\gamma_0^k z)| & \leq \sum_{\ell_{\gamma_0^k}(Y,z)\leq 1}\int_{d_{\bbH}(z,\gamma_0^k z)}^{\infty}\frac{\left|\check{f}_{\Lambda_0}'(s)\right|}{\sqrt{\cosh s -\cosh d_{\bbH}(z,\gamma_0^k z)}} ds\\
        &\lesssim \sum_{\ell_{\gamma_0^k}(Y,z)\leq 1}  \Lambda_0^{3/2} \lesssim \Lambda_0^{3/2} \delta_0^{-1},
        \end{split}
    \end{equation}
    where the second to the last step follows from Lemma~\ref{lem:unif-K} for $h(x)=x$.
   From~\eqref{eq:upper-pre-trace}, ~\eqref{eq:unif-bound-1}, and~\eqref{eq:K-bound}, we conclude \eqref{eq:unif-fn}.

   2. We now show \eqref{eq:expect-fn-sq-lem}. 
We have 
\begin{equation}\label{eq:int_exp}
\begin{split}
        &\frac{1}{\Vol(Y)}\int_{Y_{\delta_0}} \left(V(h\circ f_{\Lambda_0})(z)-S(h\circ f_{\Lambda_0})(z)\right)^2 dz\\
        & \quad \leq 2\sup_{z\in Y_{\delta_0}}
        \left(V(h\circ f_{\Lambda_0})(z)\right)^2+2\sup_{z\in Y_{\delta_0}}\left(S(h\circ f_{\Lambda_0})(z)\right)^2.
\end{split}
\end{equation}
Using ~\eqref{eq:unif-fn} and \eqref{eq:unif-bound-1}, we can bound the first term on the right-hand side by $C\Lambda_0^{12}\delta_0^{-8}\|\tilde{h}\|^8$. For the second term, using Lemma~\ref{lem:margulis}, suppose $\gamma_0(z)$ is the only possible short geodesic loop based at $z$. By Lemma~\ref{lem:geodesic_loop_bound}  and Lemma~\ref{lem:unif-K}, for $z \in Y_{\delta_0}$, we have
\begin{equation*}\label{eq:shf-bound}
\begin{split}
    |S(h\circ f_{\Lambda_0})(z)|&\lesssim \sum_{\gamma\in \Gamma_0}\left(\sup_{z\in Y_{\delta_0}} \sum_{k\in \bbZ\setminus\{0\}
    }K_{(h\circ f_{\Lambda_0})^{\vee}} (z,\gamma^{k}z)\right)^4\\
    &\lesssim \Lambda_0^4\|\tilde{h}\|^4\sum_{\gamma \neq \gamma_0(z) \in \Gamma_0}\left(\sum_{k \in \bbZ\setminus\{0\}} \exp\left(-d_{\bbH}(z,\gamma^{k}z)/2\right)\right)^4\\
    &+\Lambda_0^6\|\tilde{h}\|^4\left(\sum_{k\in \mathbb{Z}\setminus\{0\}}\exp\left({-d_{\mathbb{H}}(z,\gamma_0(z)^k z)/2}\right)\right)^4\\
    &\lesssim \delta_0^{-4}\Lambda_0^6\|\tilde{h}\|^4.
\end{split}
\end{equation*}
Therefore, we conclude that \eqref{eq:int_exp} $\lesssim \Lambda_0^{12}\delta_0^{-8}\|\tilde{h}\|^8$.
\end{proof}

Using Lemma~\ref{lem:unif-fn} and the Markov brothers' inequality ~\eqref{eq:markov}, we can now prove Proposition~\ref{prop:expect-fn}. 

\begin{proof}[Proof of Proposition~\ref{prop:expect-fn}]

Using the Selberg pre-trace formula~\eqref{eq:pretrace-1}, we rewrite the quantity we wish to bound:
\begin{equation}\label{eq:expect-fn-sq}
\begin{split}
    &\bbE\left[\int_{Y_{\delta_0}} \left(V(h\circ f_{\Lambda_0})(z)-S(h\circ f_{\Lambda_0})(z)\right)^2 dz\right]\\
    &=\sum_{(\gamma_1,\gamma_2,\gamma_3,\gamma_4)\in \widetilde{\Gamma}^4}\sum_{(\gamma_5,\gamma_6,\gamma_7,\gamma_8)\in \widetilde{\Gamma}^4}\mathbb{E} \left[ \int_{Y_{\delta_0}} \prod_{\ell=1}^{8}K_{(h\circ f_{\Lambda_0})^{\vee}}(z,\gamma_{\ell}z) dz\right],
\end{split}
\end{equation}
where $ \widetilde{\Gamma}^4$ is defined in \eqref{eq:ttilde-gamma-def}.

We apply the polynomial approximation \eqref{eq:expan-g-1}  to the right-hand side of \eqref{eq:expect-fn-sq}. By Lemma~\ref{lem:unif-K}, we have for some $p,p_1\leq C_0 q$,
\begin{equation*}
    \left|\eqref{eq:expect-fn-sq} - \frac{1}{Q_q(1/n)}\sum_{j=1}^{p_1}f_j n^{-j}\right|\lesssim (C_0p)^{\kappa p}n^{-p-1}\Lambda_0^{12}\|\tilde{h}\|^8,\quad n\geq (C_0p)^{\kappa}.
\end{equation*}

By the uniform bound Lemma~\ref{lem:unif-fn}, for $n\geq (C_0 p)^{\kappa}$, the polynomial $P(x):=\sum_{j=1}^{p_1} f_j x^j$ satisfies
 \begin{equation*}
     \left|\frac{1}{n}P(1/n)\right| \lesssim \Lambda_0^{12}\delta_0^{-8}\|\tilde{h}\|^8.
 \end{equation*}
 Recall that $Q_q(1/n) \in [C_0^{-1}, C_0]$ for $n \geq (C_0 p)^\kappa$.
 Thus by the Markov brothers' inequality ~\eqref{eq:markov}, we have
 \begin{equation*}
     \frac{P(1/n)}{ Q_{q}(1/n)}\lesssim \frac{1}{n} \|(xP(x))''\|_{[0,\frac{1}{2(C_0p)^{\kappa}}]}\lesssim \frac{q^{4\kappa}\Lambda_0^{12}}{n}\delta_0^{-8}\|\tilde{h}\|^8,\quad \text{for } n\geq 2(C_0p)^{\kappa}.
 \end{equation*}
Thus we conclude \eqref{eq:expect-fn-prop} for $n\geq 2C_0^{2\kappa}q^{\kappa} \geq 2(C_0p)^{\kappa}$. On the other hand, \eqref{eq:expect-fn-prop} follows from Lemma~\ref{lem:unif-fn} when $n\leq 2C_0^{2\kappa} q^{\kappa}$. This completes the proof of Proposition~\ref{prop:expect-fn}.
\end{proof}

\section{Proof of Theorem~\ref{thm:lambda-g}} \label{section:adapt_polynomial_method}
In this section, we prove Theorem~\ref{thm:lambda-g}, omitting details  
when the proof follows exactly as the proof of Theorem~\ref{thm:eigen-g}.   Setting $\Lambda = \Lambda_0$, let $f_{\Lambda}$ be the function defined in \S\ref{subsection:test}. Let $h(x)=x\tilde{h}(x)$ be a polynomial of fixed degree $q$.
We assume
\begin{equation}\label{eq:c0-assumption-2}
\Lambda \geq 1/4, \quad 2c_0 < \epsilon_0,
\end{equation}
where $\epsilon_0$ is the Margulis constant in Lemma~\ref{lem:margulis}.

\subsection{Proof of Theorem~\ref{thm:lambda-g}} \label{subsection:proof_of_main_thm}

We will prove Theorem~\ref{thm:lambda-g} assuming the following Proposition.

\begin{prop}
\label{prop:expectation_squared}
There exists $C=C(\kappa,C_0)>0$ such that
\begin{equation}\label{eq:expect-sq}
    \bbE\left[\frac{1}{n}\tr \left(\left(h\circ f_{\Lambda}\right)\left(\sqrt{\Delta_{Y}-1/4}\right)\right) -\frac{\Vol(Y)}{2 \pi n }\int_{0}^{\infty} \left(h\circ f_{\Lambda}\right)(r) r\tanh(\pi r) dr \right]^2 \leq C\frac{\Lambda^2q^{2\kappa}}{n}\|\tilde{h}\|^2.
\end{equation}
\end{prop}

We delay the proof of  Proposition~\ref{prop:expectation_squared} to \S \ref{subsection:proof_of_prop}.

Similarly to Corollary~\ref{cor:expect-fn-prop}, we can rewrite~ \eqref{eq:expect-sq} to hold for general smooth $h(x)=x \tilde{h}(x)$. 

\begin{cor}\label{cor:smooth_eigenvalue}
For smooth functions $h(x)=x \tilde{h}(x)$, $\tilde{h}(x) \in C^{\infty}([0,f(i/2)])$, we have
\begin{equation}\label{eq:expect-general}
\begin{split}
&\bbE\left[\frac{1}{n}\tr \left(\left(h\circ f_{\Lambda}\right)\left(\sqrt{\Delta_{Y}-1/4}\right)\right) -\frac{\Vol(Y)}{2 \pi n }\int_{0}^{\infty} \left(h\circ f_{\Lambda}\right)(r) r\tanh(\pi r) dr \right]^2\\
&\quad \leq \frac{C\Lambda^2}{n}\|\tilde{h}\|_{C^{\kappa+1}}^2.
\end{split}
\end{equation}
\end{cor}

\begin{proof}
Let $\tilde{h}(x)$ be a polynomial of degree $q-1$. We write $\tilde{h}(x)=\sum_{j=0}^{q-1} a_j T_j(2f(i/2)^{-1}x-1)$, where $T_j$ is the $j$-th Chebyshev polynomial as in \cite[(4.1)]{CGVTvH24}. Then we have
\begin{equation*}
    \begin{split}
        &\bbE\left[\frac{1}{n}\tr \left(\left(h\circ f_{\Lambda}\right)\left(\sqrt{\Delta_{Y}-1/4}\right)\right) -\frac{\Vol(Y)}{2 \pi n }\int_{0}^{\infty} \left(h\circ f_{\Lambda}\right)(r) r\tanh(\pi r) dr \right]^2 \\
        &\leq C\frac{\Lambda^2}{n}\left(\sum_{j=0}^{q-1}  (j+1)^{\kappa}|a_j|\right)^2\\
        &\leq C\frac{\Lambda^2}{n}\|\tilde{h}\|_{C^{\kappa+1}}^2,
    \end{split}
\end{equation*}
where the first step is due to the Minkowski inequality and \eqref{eq:expect-sq}, and the second step is due to \cite[Corollary 4.5]{CGVTvH24}. By approximating smooth functions by polynomials, we note that \eqref{eq:expect-general} holds for general smooth functions $h(x)=x \tilde{h}(x)$, $\tilde{h}(x) \in C^{\infty}([0,f(i/2)])$.
\end{proof}

\begin{proof}[Proof of Theorem~\ref{thm:lambda-g}]
1. Let $\alpha_0=\frac{1}{4(\kappa+1)}$. By \cite[Lemma 3.3]{DZ16}, we can take a smooth cutoff $h_{\Lambda}(x)$ such that
\begin{equation*}\label{eq:hLambdaeps}
    h_{\Lambda}(x) =\left\{\begin{array}{ll}
    1,  &   x\in [f_{\Lambda}(\sqrt{\Lambda-1/4}),f(i/2)], \\
    0,     &   x\in [0,f_{\Lambda}(\sqrt{\Lambda+\Lambda n^{-\alpha_0}-1/4})]
    \end{array}\right.
\end{equation*}
and moreover,
\begin{equation}\label{eq:tildeh}
    0\leq h_{\Lambda}(x)\leq 1,\quad h_{\Lambda}(x)=x\tilde{h}_{\Lambda}(x),\quad |\tilde{h}_{\Lambda}^{(j)}(x)|\leq C_j n^{j\alpha_0},\quad x\in [0,f(i/2)],
\end{equation}
where the last estimate follows from an argument similar to~\eqref{eq:MVT}.

Applying Corollary~\ref{cor:smooth_eigenvalue}, we have
 \begin{equation*}
    \begin{split}
        &\bbE\left[\frac{1}{n}\tr \left(\left(h_{\Lambda}\circ f_{\Lambda}\right)\left(\sqrt{\Delta_{Y}-1/4}\right)\right) -\frac{\Vol(Y)}{2 \pi n }\int_{0}^{\infty} \left(h_{\Lambda}\circ f_{\Lambda}\right)(r) r\tanh(\pi r) dr \right]^2 \\
        &\leq C\frac{\Lambda^2}{n}\|\tilde{h}_{\Lambda}\|_{C^{\kappa+1}}^2 \leq C\frac{\Lambda^2}{n} n^{2(\kappa+1)\alpha_0} \leq C\Lambda^{2} n^{-1/2},
    \end{split}
\end{equation*}
where the third step follows from \eqref{eq:tildeh}.
 
 By Chebyshev's inequality, there is a subset  $\Omega_g(\Lambda)$ of the random family of hyperbolic surfaces of genus $g$ with probability (where $C_1>0$ will be determined later in \eqref{eq:prob-sum})
 \begin{equation}\label{eq:prob}
    \mathbb{P}(\Omega_g(\Lambda))\geq 1-(C_1\Lambda)^{-2}n^{-1/4},
\end{equation} such that for $Y\in \Omega_g(\Lambda)$,
\begin{equation}\label{eq:expect-cheby}
    \left|\frac{1}{n}\tr \left(\left(h_{\Lambda}\circ f_{\Lambda}\right)\left(\sqrt{\Delta_{Y}-1/4}\right)\right) -\frac{\Vol(Y)}{2\pi n}\int_{0}^{\infty} \left( h_{\Lambda}\circ f_{\Lambda}\right)(r) r \tanh(\pi r) dr \right|\lesssim  \Lambda^2  n^{-1/8}.
\end{equation}
2. Recall the eigenvalue counting function $N_Y(\Lambda) \coloneqq \# \{\lambda_j(Y) \leq \Lambda\}$. By \eqref{eq:expect-cheby}, for $Y \in \Omega_g(\Lambda)$, we have
\begin{equation}\label{eq:upper}
\begin{split}
    \frac{1}{n} N_{Y}(\Lambda) &\leq \frac{\Vol(Y)}{2\pi n}\int_{0}^{\infty} \left( h_{\Lambda}\circ f_{\Lambda}\right)(r) r \tanh(\pi r) dr + C\Lambda^2  n^{-1/8}\\
    &\leq \frac{\Vol(Y)}{2\pi n}\int_{0}^{\sqrt{\Lambda+\Lambda  n^{-\alpha_0}-1/4}} r \tanh(\pi r) dr + C \Lambda^2  n^{-1/8}\\
    &\leq \frac{\Vol(Y)}{2\pi n}\int_{0}^{\sqrt{\Lambda-1/4}} r \tanh(\pi r) dr + C\Lambda n^{-\alpha_0}  +C\Lambda^2  n^{-1/8}.
\end{split}
\end{equation}
In particular, we know \eqref{eq:weyl-g} for $\Lambda\leq 1/4+n^{-\alpha_0}$ and $Y \in \Omega_g(\Lambda)$. On the other hand, for $\Lambda\geq 1/4+n^{-\alpha_0}$, by \eqref{eq:expect-cheby} we have for $Y\in \Omega_g(\Lambda-\Lambda n^{-\alpha_0})$
\begin{equation}\label{eq:lower}
\begin{split}
   \frac{1}{n} N_{Y}(\Lambda) &\geq \frac{\Vol(Y)}{2\pi n}\int_{0}^{\infty} \left( h_{\Lambda-\Lambda n^{-\alpha_0} }\circ f_{\Lambda-\Lambda n^{-\alpha_0}}\right)(r) r \tanh(\pi r) dr - C \Lambda^2  n^{-1/8}\\
   &\geq \frac{\Vol(Y)}{2\pi n}\int_{0}^{\sqrt{\Lambda-\Lambda n^{-\alpha_0}-1/4}} r \tanh(\pi r) dr - C \Lambda^2  n^{-1/8}\\
   &\geq \frac{\Vol(Y)}{2\pi n}\int_{0}^{\sqrt{\Lambda-1/4}} r \tanh(\pi r) dr -C\Lambda n^{-\alpha_0} - C \Lambda^2  n^{-1/8}.
    \end{split}
\end{equation}
Combining \eqref{eq:upper} and \eqref{eq:lower}, we conclude \eqref{eq:weyl-g} for $Y\in \Omega_g(\Lambda)\cap \Omega_g(\max(\Lambda-\Lambda n^{-\alpha_0},1/4))$. 

3. Finally, we find a set $\Omega_g$ that does not depend on $\Lambda$. Take $\Lambda(j) := 2^{j-2}$ for $j\geq 0$. Let
\begin{equation*}
    \Lambda(j,\ell) := \Lambda(j)+\ell n^{-\alpha_0}\Lambda(j)\in [\Lambda(j),\Lambda(j+1)],\,\,  0\leq\ell \leq L_{j}:= \left\lfloor n^{\alpha_0} \frac{\Lambda(j+1)-\Lambda(j)}{\Lambda(j)}\right\rfloor ,
\end{equation*}
and
\begin{equation*}
    \Omega_g:=\bigcap_{j=0}^{\infty}\bigcap_{\ell=0}^{L_j}\Omega_g(\Lambda(j,\ell))\cap \Omega_g(\max(\Lambda(j,\ell)-\Lambda(j,\ell) n^{-\alpha_0},1/4)) .
\end{equation*}
Taking $C_1$ in \eqref{eq:prob} sufficiently large, we have
\begin{equation}\label{eq:prob-sum}
\begin{split}
    \mathbb{P}(\Omega_g)&\geq 1-C\sum_{j=0}^{\infty}\sum_{\ell=0}^{L_j} (C_1\Lambda(j,\ell))^{-2}n^{-1/4}\\
    &\geq 1-C\sum_{j=0}^{\infty} (C_1\Lambda(j))^{-2}C n^{\alpha_0} n^{-1/4}\\
    &\geq 1-n^{-1/10}. 
\end{split}
\end{equation}
Now \eqref{eq:weyl-g} holds for all $\Lambda(j,\ell)$ and $Y \in \Omega_g$. It follows that \eqref{eq:weyl-g} holds for all $\Lambda$.
\end{proof}

\subsection{Proof of Proposition~\ref{prop:expectation_squared}}
\label{subsection:proof_of_prop}
Here, $h(x) = x \tilde{h}(x)$ is a degree $q$ polynomial.

We  prove an analogue of Lemma~\ref{lem:unif-fn}.
\begin{lem}\label{lem:triv_bdd}
We have the bound
\begin{equation}\label{eq:triv_bdd}
\bbE\left[\frac{1}{n}\tr \left(\left(h\circ f_{\Lambda}\right)\left(\sqrt{\Delta_{Y}-1/4}\right)\right) -\frac{\Vol(Y)}{2 \pi n }\int_{0}^{\infty} \left(h\circ f_{\Lambda}\right)(r) r\tanh(\pi r) dr \right]^2  \lesssim \Lambda^2\|\tilde{h}\|^2.
\end{equation}
\end{lem}

\begin{proof}

1. We first prove the initial bound
    \begin{equation}\label{eq:unif-bound}
        \frac{1}{n}\left|\tr h \circ f_{\Lambda}(\sqrt{\Delta_{Y}-1/4})\right|\lesssim \Lambda(1+|\log \injrad_1(Y)|)\|\tilde{h}\|.
    \end{equation}

We know $h\circ f_{\Lambda}= f_{\Lambda}\cdot (\tilde{h}\circ f_{\Lambda}) $. Since the norm of $\tilde{h}\circ f_{\Lambda}(\sqrt{\Delta_{Y}-1/4})$ is bounded by $\|\tilde{h}\|$, it suffices to show
\begin{equation}\label{eq:it-suffices-to-show}
     \frac{1}{n}\left|\tr f_{\Lambda}(\sqrt{\Delta_{Y}-1/4})\right|\lesssim \Lambda (1+|\log\injrad_1(Y)|).
\end{equation}

   Due to the Selberg trace formula \eqref{eq:selberg_trace_formula}, \eqref{eq:it-suffices-to-show} follows from \eqref{eq:unif-bound-1}
    and showing
    \begin{equation}\label{eq:unif-bound-2}
       \frac{1}{n}\sum_{\gamma \in \calP(Y)} \sum_{k=1}^\infty \frac{\ell_\gamma(Y)}{2 \sinh\left(\frac{k \ell_\gamma(Y)}{2}\right)} \check{f}_{\Lambda}(k \ell_\gamma(Y))\lesssim \Lambda^{1/2}(1+|\log\injrad_1(Y)|).
    \end{equation}
    We now prove \eqref{eq:unif-bound-2}.
    By 
    ~\eqref{eq:support}, we only need to consider geodesics of length $\leq c_0\Lambda^{-1/2}$. From ~\eqref{eq:c0-assumption-2} and \cite[Theorem 4.1.6]{Bu10} there are at most $3g-3$ such primitive geodesics  and from~\eqref{eq:support},
    $$\sum_{k=1}^\infty \frac{\ell_\gamma(Y)}{2 \sinh\left(\frac{k \ell_\gamma(Y)}{2}\right)} \check{f}_{\Lambda}(k \ell_\gamma(Y)) \lesssim \Lambda^{1/2}(1+|\log\injrad_1(Y)|).$$
 From~\eqref{eq:g-n-2}, we conclude~\eqref{eq:unif-bound}.

2. We now show~\eqref{eq:triv_bdd}. We have
\begin{equation*}
  \begin{split} 
  & \mathbb{E}\left(\frac{1}{n}\tr \left(\left(h \circ f_{\Lambda}\right)\left(\sqrt{\Delta_{Y}-1/4}\right)\right)-\frac{\Vol(Y)}{2\pi n}\int_{0}^{\infty} \left(h \circ f_{\Lambda}\right)(r) r\tanh(\pi r) dr\right)^2\\
  &\leq 2\mathbb{E}\left(\frac{1}{n}\tr \left( \left(h \circ f_{\Lambda}\right)\left(\sqrt{\Delta_{Y}-1/4}\right)\right)\right)^2+2\bbE\left[\frac{\Vol(Y)}{2\pi n}\int_{0}^{\infty} \left(h \circ f_{\Lambda}\right)(r) r\tanh(\pi r) dr\right]^2\\
  &\lesssim \Lambda^2\|\tilde{h}\|^2,
\end{split}
\end{equation*}
where the final inequality follows from the fact that the first term is bounded by \eqref{eq:unif-bound} and \eqref{eq:sys-2} and the fact that the second term is uniformly bounded by 
\begin{equation*}
    \left(\frac{\Vol(Y)}{2\pi n}\int_{0}^{\infty} f_{\Lambda}(r) r \tanh(\pi r) dr \right)^2\|\tilde{h}\|^2 \lesssim \Lambda^2 \|\tilde{h}\|^2.\qedhere
\end{equation*}
\end{proof}

Using Lemma~\ref{lem:triv_bdd} and the Markov brothers' inequality, we now prove Proposition~\ref{prop:expectation_squared}.

\begin{proof}[Proof of Proposition~\ref{prop:expectation_squared}]

By the Selberg trace formula \eqref{eq:selberg_trace_formula}, we have 
\begin{equation}
\label{eq:function_with_expectation}
\begin{split}
&\bbE\left[\frac{1}{n}\tr \left(\left(h\circ f_{\Lambda}\right)\left(\sqrt{\Delta_{Y}-1/4}\right)\right) -\frac{\Vol(Y)}{2 \pi n }\int_{0}^{\infty} \left(h\circ f_{\Lambda}\right)(r) r\tanh(\pi r) dr \right]^2\\
&=\bbE\left[\frac{1}{n}\sum_{\gamma \in \calP(Y)} \sum_{k=1}^\infty \frac{\ell_\gamma(Y)}{2 \sinh\left(\frac{k \ell_\gamma(Y)}{2}\right)} (h \circ f_{\Lambda})^{\vee} (k \ell_\gamma(Y)) \right]^2.
\end{split}
\end{equation} 

By \eqref{eq:eigenvalue_condition}, there are $p,p_1\leq C_0 q$ such that
\begin{equation*}
    \left|\eqref{eq:function_with_expectation} -\frac{1}{n^2 Q_q(1/n)}\sum_{j=0}^{p_1} f_j n^{-j}\right| \lesssim (C_0 p)^{\kappa p} n^{-p-2}\Lambda^2\|\tilde{h}\|^2,\quad n\geq (C_0 p)^{\kappa}.
\end{equation*}
Let $P(x)=\sum_{j=0}^{p_1} f_j x^j$. From Lemma~\ref{lem:triv_bdd}, we know
\begin{equation}
\label{eq:f_bound}
\frac{1}{n^2} |P(1/n)|\lesssim \Lambda^2\|\tilde{h}\|^2,\quad n\geq  (C_0 p)^{\kappa}.
\end{equation}

Using the Markov's brothers' inequality~\eqref{eq:markov} and \eqref{eq:f_bound}, we know that
$$\|(x^2P(x))'\|_{[0,\frac{1}{2(C_0 p)^{\kappa}}]} \lesssim q^{2\kappa } \sup_{n\geq (C_0p)^{\kappa}}\frac{1}{n^2}|P(1/n)| \lesssim q^{2\kappa}\Lambda^2 \|\tilde{h}\|^2.$$

Recall that $Q_q(1/n) \in [C_0^{-1}, C_0]$ for $n \geq (C_0 p)^\kappa$. We then take
the Taylor expansion to see
$$\left|\frac{1}{n^2} \frac{P(1/n)}{Q_{q}(1/n)}\right| \lesssim \frac{1}{n}\|(x^2 P(x))'\|_{[0,\frac{1}{2(C_0p)^{\kappa}}]} \lesssim \frac{ q^{2\kappa}\Lambda^2}{n}\|\tilde{h}\|^2,\quad n \geq 2(C_0p)^{\kappa}. $$
In other words, $\eqref{eq:function_with_expectation}\lesssim \frac{q^{2\kappa}\Lambda^2}{n}\|\tilde{h}\|^2$ for $n\geq 2(C_0p)^{\kappa}$. This shows \eqref{eq:expect-sq} for $n\geq 2 C_0^{2\kappa}q^{\kappa}\geq 2(C_0 p)^{\kappa}$. On the other hand, for $n\leq 2 C_0^{2\kappa}q^{\kappa}$, \eqref{eq:expect-sq} follows from Lemma~\ref{lem:triv_bdd}. 
\end{proof}

\section{Polynomial expansion}\label{section:polynomial_expansion}

In this section, we prove Theorem~\ref{thm:eigenfn} and Theorem~\ref{thm:cover} by showing the random cover model satisfies the conditions of Theorems~\ref{thm:eigen-g}~and~\ref{thm:lambda-g}. The key step is an expansion of the trace for random permutation representations of the surface group, which we prove in the following subsection. Recall the preliminaries for random covers given in \S\ref{subsection:surface_group}.

\subsection{Expansion of the trace}\label{subsection:expand-trace}

In \cite[\S 4]{MPvH25}, the authors examined  $\bbE[\Tr(\rho(\gamma)]$ for $\gamma\in \Gamma\setminus\{\id\}$.
We adapt the arguments from 
\cite[\S 4]{MPvH25} to prove a rational function approximation of $$\bbE\left[\Tr(\rho(\gamma_1))\Tr(\rho(\gamma_2))\right] \quad \text{and} \quad \bbE \left[\sum_{i=1}^n \prod_{\ell=1}^{8}\rho_{ii}(\gamma_\ell)\right] .$$

We begin with the following rational function approximation, which we will use to show the polynomial approximation condition \eqref{eq:eigenvalue_condition} in Theorem~\ref{thm:lambda-g}. This is the simpler of the two approximations; we present it first for the reader's convenience.

\begin{lem}\label{lem:fixed_point_estimte}
There exists $C=C(g_0)>2$ such that for $n^{-1} \in [0, (C q)^{-C}]$ and $\gamma_1,\gamma_2 \in \Gamma \setminus \{\id\}$ with $|\gamma_1|+|\gamma_2|\leq q$, 
$$\bbE\left[\Tr(\rho(\gamma_1))\Tr(\rho(\gamma_2))\right] = \frac{Q_{\gamma_1, \gamma_2}(1/n)}{Q_\id(1/n)} +  O \left((C q)^{C q} n^{-q}\right),$$
where $Q_{\gamma_1, \gamma_2}$, $Q_\id$ are polynomials with 
\begin{equation*}
    \deg (Q_{\gamma_1, \gamma_2})\leq 9 q(4g_0+1),\quad \deg(Q_\id) \leq 9 q(4g_0+1)+1,
\end{equation*} and  $Q_{\id}(1/n) \in [C^{-1},C]$ for $n\geq (Cq)^{C}$. 
\end{lem}

\begin{proof}

1. Recall that $\bbX_{g_0, n} \coloneqq \Hom(\Gamma, S_n)$ and $\Gamma$ is a surface group with $2g_0$ generators:
$\Gamma = \langle a_1, a_2, \ldots, a_{2g_0-1}, a_{2g_0} | [a_1, a_2] \cdots [a_{2g_0-1}, a_{2g_0}]=1\rangle$.

Let $\gamma_1, \gamma_2 \in \Gamma$. We write $\gamma_\ell=a_{i_{1,\ell}}^{\alpha_{1,\ell}}a_{i_{2,\ell}}^{\alpha_{2,\ell}} \cdots a_{i_{m_{\ell},\ell}}^{\alpha_{m_{\ell},\ell}}$, where $i_{k,\ell} \in \{1,2,\ldots,2g_0\}$, $\alpha_{k,\ell}\in \{-1, 1\}$, and $m_{\ell}=|\gamma_{\ell}|$ is the word length of $\gamma_{\ell}$. We assume that the above expressions of $\gamma_{\ell}$, written in terms of the generators, are cyclically reduced.  We do not assume that the corresponding closed geodesics are primitive; they may be powers of primitive geodesics.

We define the graph $C_{\gamma_1, \gamma_{2}}$ to be the graph of $2$ disjoint cycles of length $m_{\ell}$, where the $\ell$-th cycle has length $m_{\ell}$ and has directed edges labeled by $a_{i_{1,\ell}}^{\alpha_{1,\ell}}, a_{i_{2,\ell}}^{\alpha_{2,\ell}}, \ldots ,  a_{i_{m_{\ell},\ell}}^{\alpha_{m_{\ell},\ell}}$. Specifically, the  $\ell$-th cycle has a directed edge labeled by $a_{i_k, \ell}$ from the $k$-th to the $(k+1)$-th vertex when $\alpha_{k,\ell} =1$ and from the $(k+1)$-th to the $k$-th vertex when  $\alpha_{k, \ell} =-1$.

Let $\calR$ denote the collection of surjective labeled-graph morphisms 
$$r: C_{\gamma_1, \gamma_2} \twoheadrightarrow W_r$$
such that
\begin{enumerate}
    \item \label{item:condition1} $W_r$ is \emph{folded}, i.e.,
 every vertex has at most one incoming $a$-labeled edge and at most one outgoing $a$-labeled edge for each $a \in \{a_1, \ldots, a_{2g_0}\}$. 
 \item \label{item:condition2} Every path in $W_r$ that spells out an element of the free group $F_{2g_0}$  in the kernel of $F_{2g_0} \rightarrow \Gamma$ is closed.
\end{enumerate}

Note that $\# \calR \leq (|\gamma_1| + |\gamma_2|)!$.

2. 
Let $\phi \in \bbX_{g_0,n}$.
We denote the Schreier graph by 
$$X_\phi \coloneqq \text{Schreier}(\{\phi(a_1), \ldots, \phi(a_{2g_0})\}, [n]).$$
This is the graph on $n$ vertices labeled $1, \ldots, n$, where there is a directed edge from $i$ to $j$ exactly when there exists $f \in \{a_1, \ldots, a_{2g_0}\}$ such that $\phi(f)[i] = j$. Such an edge is labeled by $f$.

By construction, it is clear that
any morphism of folded labeled graphs $C_{\gamma_1, \gamma_2} \rightarrow X_\phi$ factors uniquely as a surjective morphism  followed by an injective one, namely
$C_{\gamma_1, \gamma_2} \twoheadrightarrow W_r \hookrightarrow X_\phi$
for unique $r \in \calR$.

For $\phi \in \bbX_{g_0,n}$, let $\text{Fix}(\phi(\gamma))$ denote the number of fixed elements of $[n]$ under the action of $\phi(\gamma)$.
Therefore,
\begin{align*}
\bbE [\tr(\rho(\gamma_1))\tr(\rho(\gamma_2))]&=\frac{1}{\#\bbX_{g_0,n}} \sum_{\phi \in \bbX_{g_0,n}} [\text{Fix}(\phi(\gamma_1))\text{Fix}(\phi(\gamma_2))]\\
&=\frac{1}{\# \bbX_{g_0,n}} \sum_{\phi \in \bbX_{g_0,n}}  [\# \text{ of morphisms } C_{\gamma_1, \gamma_2}\rightarrow X_\phi]\\
&=\frac{1}{\# \bbX_{g_0,n}} \sum_{r \in \calR} \sum_{\phi \in \bbX_{g_0,n}}  [\# \text{ of injective morphisms } W_r \hookrightarrow X_\phi].
\end{align*}

In other words, let
$$\bbE_n^{\text{emb}}(W_r) \coloneqq \bbE_{\phi \in \bbX_{g_0,n}}[\# \text{ of injective morphisms } W_r \hookrightarrow X_\phi],$$
$$\bbE[\tr(\rho(\gamma_1))\tr(\rho(\gamma_2))] = \sum_{r \in \calR} \bbE_n^{\text{emb}}(W_r).$$

3. We now estimate $\bbE_n^{\text{emb}}(W_r)$, closely following the arguments of \cite[\S 4.4]{MPvH25}.

Let $\zeta(s; S_n)$ be the Witten zeta function of $S_n$. We defer to  \cite[pg. 18]{MPvH25} for the formal definition. Here we just use the facts that $\zeta(s; S_n) \rightarrow 2$ as $n\rightarrow \infty$ and from \cite[(1.3)]{MP23}, $\# \bbX_{g_0,n} = (\#S_n)^{2g_0-1} \zeta(2g_0-2, S_n)$.

As noted above, \cite{MPvH25} uses $C_\gamma$ instead of $C_{\gamma_1, \gamma_2}$. However, their proofs hold as written when $C_\gamma$ is replaced by $C_{\gamma_1, \gamma_2}$. Therefore, we conclude from  \cite[pg. 26]{MPvH25} that for $|\gamma_1|+|\gamma_2| \leq q$ and $n \geq 28 q^2$,
$$ \bbE^{\text{emb}}_{n}(W_r) = \frac{2}{\zeta(2g_0-2, S_n)}\left(\frac{p_r(n)}{(n)_{9q}^{1+4g_0}} +  O \left((Cq)^{Cq} n^{-q}\right)\right).$$
Here 
$(n)_{9q}=n(n-1)\cdots(n-9q+1)$ is the falling Pochhammer symbol 
and $p_r$ is a polynomial in $n$ with $\deg(p_r) \leq 9q(4g_0+1)+2$.

Since $\#\calR \leq (|\gamma_1| + |\gamma_2|)!$, we conclude 
\begin{equation}
\label{eq:fixed_point_intermediate}
\bbE[\tr(\rho(\gamma_1))\tr(\rho(\gamma_2))] = \frac{2}{\zeta(2g_0-2, S_n)}\left(\frac{p_{\gamma_1, \gamma_2}(n)}{(n)_{9q}^{1+4g_0}} +  O \left((Cq)^{Cq} n^{-q}\right)\right),
\end{equation}
where $p_{\gamma_1, \gamma_2}$ is a polynomial in $n$ with $\deg(p_{\gamma_1, \gamma_2}) \leq 9q(4g_0+1)+2$. 

4. It remains to transform \eqref{eq:fixed_point_intermediate} into a rational function in $1/n$ with the same error term.

Let 
$$g_q(t) \coloneqq \prod_{k=0}^{9q-1}(1-kt)^{1+4g_0}.$$
For $t \coloneqq n^{-1}$ and a polynomial $P_{\gamma_1, \gamma_2}(t)$ of degree at most $\deg(p_{\gamma_1, \gamma_2})$,
\begin{align*}
&\frac{p_{\gamma_1, \gamma_2}(n)}{(n)_{9q}^{1+4g_0}} = t^{9q(1+4g_0)+2 - \deg(p_{\gamma_1, \gamma_2})}\frac{P_{\gamma_1, \gamma_2}(t)}{t^2g_q(t)} 
\\
&\eqqcolon \frac{Q_{\gamma_1, \gamma_2}(t)+a_{-1}(\gamma_1,\gamma_2)t^{-1}+a_{-2}(\gamma_1,\gamma_2)t^{-2}}{ g_q(t)}.
\end{align*}

Note that for 
$t \in [0, (Cq)^{-C}]$, $C>2$, we have $C^{-1}\leq g_q(t) \leq C$. By \cite[Proposition 3.1]{Naud25}, we have
\begin{equation}\label{eq:Q_bound-1}
    \bbE [\tr(\rho(\gamma_1))\tr(\rho(\gamma_2))] \lesssim_q 1,\quad n\to \infty.
\end{equation}
Therefore, $a_{-1}(\gamma_1,\gamma_2)=a_{-2}(\gamma_1,\gamma_2)=0$ and 
\begin{equation*}
\label{eq:Q_bound}
Q_{\gamma_1, \gamma_2}(t) = O(1), \quad \text{for } t = n^{-1} \in \left[0, (Cq)^{-C}\right].
\end{equation*}

Also note that $\deg(Q_{\gamma_1, \gamma_2}) \leq 9 q (4g_0+1)$. 

From \cite[(4.19)]{MPvH25}, for a polynomial $Q_\id(t)$ with $\deg(Q_\id) \leq 9q(4g_0+1)+1$,
$$\frac{2}{\zeta(2g_0-2, S_n)} = \frac{g_q(1/n)}{Q_\id(1/n)}\left(1 +  O \left((Cq)^{Cq}n^{-q-1}\right) \right).$$

Therefore,
\begin{align*}
&\bbE [\tr(\rho(\gamma_1))\tr(\rho(\gamma_2))] \\
& = \frac{2}{\zeta(2g_0 - 2, S_n)} \left(\frac{Q_{\gamma_1, \gamma_2}(1/n)}{g_q(1/n)} +  O \left((Cq)^{Cq} n^{-q}\right) \right)\\
& =\frac{Q_{\gamma_1, \gamma_2}(1/n)}{Q_\id(1/n)}\left(1 +  O \left((Cq)^{Cq}n^{-q-1}\right) \right) +  O \left((Cq)^{Cq} n^{-q}\right)\\
&= \frac{Q_{\gamma_1, \gamma_2}(1/n)}{Q_\id(1/n)} +  O \left((Cq)^{Cq} n^{-q}\right),
\end{align*}
which completes the proof.
\end{proof}

We now prove the analogous rational expansion needed for the polynomial approximation condition~\eqref{eq:expan-g-1} in Theorem~\ref{thm:eigen-g}.
\begin{lem}\label{lem:fix-2}
There exists $C=C(g_0)>2$ such that for $n^{-1} \in [0, (C q)^{-C}]$ and $\gamma_\ell\in \Gamma \setminus \{\id\}$ with $|\gamma_1|+|\gamma_2|+\cdots+|\gamma_{8}|\leq q$, we have
    \begin{equation}\label{eq:fix_1}
\mathbb{E}\left[\sum_{i=1}^n\prod_{\ell=1}^{8}\rho_{ii}(\gamma_\ell)\right] = \frac{Q_{\gamma_1,\gamma_2,\ldots,\gamma_8}(1/n)}{ Q_{\id}(1/n)}+O((Cq)^{C q}n^{-q-1}).
\end{equation}
where $Q_{\gamma_1, \gamma_2,\ldots,\gamma_8}$, $Q_\id$ are polynomials with degree $\leq 9 (q+1)(4g_0+1)+1 $ and  $Q_{\id}(1/n) \in [C^{-1},C]$ for $n\geq (Cq)^{C}$.

Moreover, if $(\gamma_1,\ldots,\gamma_8) \in (\tilde{\Gamma}^4)^2$  
then 
\begin{equation}\label{eq:fix_2}
    Q_{\gamma_1,\ldots,\gamma_8}(0)=0.
\end{equation}
\end{lem}

\begin{proof}
    Note that $\sum_{i=1}^n \prod_{\ell=1}^8 \rho_{ii}(\gamma_\ell)$  
counts the number of common fixed points of $\rho(\gamma_{\ell})$, $\ell = 1, \ldots, 8$. The expansion \eqref{eq:fix_1} follows from the argument of Lemma~\ref{lem:fixed_point_estimte}, with the following changes. We replace the graph $C_{\gamma_1,\gamma_2}$ with $\tilde{C}_{\gamma_1,\ldots,\gamma_8}$. Here,  $\tilde{C}_{\gamma_1,\ldots,\gamma_8}$ is the quotient of $C_{\gamma_1, \ldots,\gamma_8}$, a graph with 8 loops corresponding respectively to $\gamma_1, \ldots, \gamma_8$, where the quotienting action identifies the first vertices of each loop. Our proof then follows exactly as that of Lemma~\ref{lem:fixed_point_estimte}. To obtain \eqref{eq:fix_2}, first note that if $(\gamma_1,\ldots,\gamma_8) \in (\tilde{\Gamma}^4)^2$,   then some pair $\gamma_j$ and $\gamma_k$ generate a free group of rank $2$. Then it suffices to replace \eqref{eq:Q_bound-1} by the following estimate:
\begin{equation*}
   \mathbb{E}\left[\sum_{i=1}^{n}\rho_{ii}(\gamma_j)\rho_{ii}(\gamma_k)\right]=\bbE\left[ \mathrm{fix}_{\langle \gamma_j,\gamma_k\rangle} \right] \lesssim_q \frac{1}{n},
\end{equation*}
where $\mathrm{fix}_{\langle \gamma_j,\gamma_k\rangle}$ denotes the number of common fixed points of $\gamma_j$ and $\gamma_k$ and the last inequality follows from \cite[Theorem 1.3]{MP23}.
\end{proof}

\subsection{Proof of Theorem~\ref{thm:eigenfn}}\label{subsection:polynomial_eigenfunction}

Using Lemma~\ref{lem:fix-2}, we show that we can apply Theorem~\ref{thm:eigen-g} to the random cover model.
\begin{prop}\label{prop:expansion-fn}
    Let $X_n$ be random covers of $X$, $L \geq 1$, and $F_1,\ldots,F_8:(0,\infty) \to \mathbb{R}$ be continuous functions with $\supp F_{\ell} \subset [0, L]$. Let $\delta_0 = \min\{\injrad(X)/2, \mathrm{arcsinh}(1)/4\}$. Then \eqref{eq:expan-g-1} holds for some $\kappa=\kappa(g_0)>2$ and $C_0=C_0(X)>2$.
\end{prop}
\begin{proof}
    In the random cover model, a point is given by $(z,i)\in \mathbb{H}\times [n]$. Suppose a fundamental domain of $X$ is given by $F_X$. The expectation in \eqref{eq:expan-g-1} becomes
    \begin{equation}\label{eq:expect-fn-expand}
    \begin{split}
        &\mathbb{E}
\left[ \int_{F_X} \sum_{i\in [n]}  \sum_{(\gamma_1, \ldots, \gamma_8) \in (\tilde{\Gamma}^4)^2} \rho_{ii}(\gamma_1) \cdots \rho_{ii}(\gamma_8) F_1(d_{\mathbb{H}}(z,\gamma_1 z) ) \cdots  F_8(d_{\mathbb{H}}(z,\gamma_8 z))  dz  \right]   \\
&=\int_{F_X} \sum_{(\gamma_1, \ldots, \gamma_8) \in (\tilde{\Gamma}^4)^2}  \sum_{i\in [n]}\mathbb{E}
\left[ \rho_{ii}(\gamma_1) \cdots \rho_{ii}(\gamma_8) \right]  F_1(d_{\mathbb{H}}(z,\gamma_1 z) ) \cdots  F_8(d_{\mathbb{H}}(z,\gamma_8 z))  dz .
    \end{split}
\end{equation}
From \cite[Lemma 2.3]{MNP22}, we have the following relationship between the length of geodesics $\ell_\gamma(X)$ and word length $|\gamma|$:
\begin{equation}
\label{eq:compare_length}
|\gamma| \leq K_1 \ell_{\gamma}(X) + K_2,
\end{equation}
where $K_1, K_2>0$ depend only on $X$ and the choice of generators. 
Using $d_{\bbH}(z,\gamma z) \geq \ell_\gamma(X)$ and \eqref{eq:compare_length}, the terms on the right-hand side of \eqref{eq:expect-fn-expand} are zero unless the word lengths of $\gamma_\ell$ satisfy $|\gamma_\ell| \leq C L$ for $\ell=1,2,\ldots,8$. We  may assume $|\gamma_\ell|\leq p/8$ by taking $p\geq 8CL$.

Thus, \eqref{eq:expan-g-1} follows from Lemma~\ref{lem:fix-2} with the remainder given by 
\begin{equation*}
    \prod_{\ell=1}^{8}\left( \sum_{d_{\mathbb{H}}(\gamma_{\ell}z, z)\leq L} \|F_{\ell}\|_{(0,L]}\right) (Cp)^{Cp}n^{-p-1} . 
\end{equation*}
Using Lemma~\ref{lem:geodesic_loop_bound}, we have 
\begin{align*}
    &\prod_{\ell=1}^8 \left(\sum_{d_{\mathbb{H}}(\gamma_{\ell}z, z)\leq L} \|F_{\ell}\|_{(0,L]} \right) (Cp)^{Cp}n^{-p-1} \\
    &\lesssim e^{8p}(Cp)^{Cp}n^{-p-1} \prod_{\ell=1}^{8}\|F_{\ell}\|_{(0,L]}\leq (C_0 p)^{\kappa p} n^{-p-1} \prod_{\ell=1}^{8}\|F_{\ell}\|_{(0,L]}.\qedhere
\end{align*}
\end{proof}

Finally, we apply Theorem~\ref{thm:eigen-g} to conclude Theorem~\ref{thm:eigenfn}.

\begin{proof}[Proof of Theorem~\ref{thm:eigenfn}]
Set $2\delta_0 =\min\{\mathrm{arcsinh}(1)/2, \injrad(X)\}$.
From Theorem~\ref{thm:eigen-g} and Proposition~\ref{prop:expansion-fn}, we conclude 
\begin{equation}\label{eq:proof1-1}
    \|u_j\|_{L^{\infty}(X_n)}\lesssim \Lambda^2 n^{-\alpha}\|u_j\|_{L^2(X_n)}.
\end{equation}
We interpolate \eqref{eq:proof1-1} with the ``trivial" estimate (see \cite{donnelly}):
    \begin{equation}\label{eq:interpo-1}
        \|u_j\|_{L^{\infty}(X_n)}\leq C\Lambda^{1/4}\|u_j\|_{L^2(X_n)},
    \end{equation}
where $C$ is uniform over all choices of covers
to conclude Theorem~\ref{thm:eigenfn}.
\end{proof}

\begin{rem}\label{rem:exchange}
The estimate \eqref{eq:linfty} implies the following bound \eqref{eq:isotropic}, which is an analogue of the graph theoretic isotropic delocalization result \cite[Corollary 1.3 (1)]{BHY}. We briefly recount the argument.

Let $a_m\in C^{\infty}(F)$ be test functions on a fundamental domain $F$ with $$\int_{F}|a_m(z)| dz\leq 1,\quad m=1,2,\ldots, M.$$ 
Let $u_j^{X_n}(z, i)$ be the $j$th eigenfunction on $X_n$.
Then the random variables
    \begin{equation*}
        Y_{i,j,m}:=\int_{F}a_m(z) u_j^{X_n}(z,i) dz,\quad i\in [n],\, j\geq 0,\, m=1,2,\ldots, M,
    \end{equation*}
    satisfy the exchangeability assumption in \cite[(8.2)]{BKY}. Let $\mathcal{A}$ be the event that \eqref{eq:linfty} holds. For any deterministic $t_i$, $i=1,2,\ldots,n$, with $\sum_{i=1}^{n} t_i=0$ and $\sum_{i=1}^{n} t_i^2\leq 1$, by \cite[Proposition 8.3]{BKY}, we have
\begin{equation*}
    \left(\bbE \left| \sum_{i=1}^{n} t_i \mathbf{1}_{\mathcal{A}}  Y_{i,j,m} \right|^p \right)^{1/p}\leq C\frac{p^2}{\log p} \left(\bbE |\mathbf{1}_{\mathcal{A}}  Y_{1,j,m}|^p\right)^{1/p}\leq C\frac{p^2}{\log p} (1+\lambda_j(X_n))^{3/2} n^{-\alpha},\quad p\geq 2.
\end{equation*}
By Chebyshev's inequality with $p=\zeta (\log \zeta)^{1/4}$ for $\zeta\gg 1$, we have
\begin{equation*}
    \mathbb{P}\left(\left| \sum_{i=1}^{n} t_i \mathbf{1}_{\mathcal{A}}  Y_{i,j,m} \right| > C (1+\lambda_j(X_n))^{3/2} \zeta^2 n^{-\alpha} \right)\leq e^{-\zeta(\log \zeta)^{1/4}}.
\end{equation*}
Taking $\zeta=\log n + \log M + \log (2+\lambda_j(X_n))$, conditional on $\mathcal{A}$ (which holds with probability at least $1-n^{-1/10}$), with very high probability,
\begin{equation}\label{eq:isotropic}
    \left| \sum_{i=1}^{n} t_i Y_{i,j,m} \right|\lesssim  (1+\lambda_j(X_n))^{3/2}(\log n + \log M +\log (2+\lambda_j(X_n))  )^2n^{-\alpha}
\end{equation}
for all $j\geq 0$ and $m=1,2,\ldots, M$.

If we repeat the same argument for the probabilistic QUE estimate in \cite[Corollary 1.3 (2)]{BHY}, we conclude an estimate with a $\sqrt{|\mathcal{X}|}n^{-2\alpha+\epsilon}$-error bound for a deterministic index set $\mathcal{X}\subset [n]$. Since $|\mathcal{X}|\leq n$, this error term is small only if $\alpha>1/4$. We note that \cite[(1.8)]{BHY} corresponds to the case that $\alpha=1/2-\epsilon$. Since our $\alpha$ in Theorem~\ref{thm:eigenfn} is smaller than $1/4$, we cannot conclude a strong result for the probabilistic QUE.
\end{rem}

\subsection{Proof of Theorem~\ref{thm:cover}} \label{subsubsection:polynomial_expansion}

Using Lemma~\ref{lem:fixed_point_estimte}, we show that we can apply Theorem~\ref{thm:lambda-g} to the random cover model.

\begin{prop}\label{prop:rational_function}
Let $X_n$ be random covers of $X$, $L \geq 1$, and $F:(0,\infty) \to \mathbb{R}$ be a continuous function with $\supp F \subset [0, L]$. Then \eqref{eq:eigenvalue_condition} holds for some $\kappa=\kappa(g_0)>2$ and $C_0=C_0(X)>2$.
\end{prop}

\begin{proof}
1. First, we claim that
\begin{equation}
\label{eq:twisted_selberg_trace}
\begin{split}
&\sum_{k_{1}, k_2} \sum_{\gamma_{1},  \gamma_{2} \in \calP(X_n)} \frac{\ell_{\gamma_1}(X_n)}{2\sinh \frac{k_1\ell_{\gamma_1}(X_n)}{2}}\frac{\ell_{\gamma_2}(X_n)}{2\sinh \frac{k_2\ell_{\gamma_2}(X_n)}{2}} F(k_1\ell_{\gamma_1}(X_n))F(k_2\ell_{\gamma_2}(X_n))\\
&=  \sum_{k_{1}, k_2} \sum_{\gamma_{1},  \gamma_{2} \in \calP(X)}  \frac{\ell_{\gamma_1}(X)}{2\sinh \frac{k_{1}\ell_{\gamma_1}(X)}{2}}\frac{\ell_{\gamma_2}(X)}{2\sinh \frac{k_{2}\ell_{\gamma_2}(X)}{2}}F(k_1\ell_{\gamma_1}(X))F(k_2\ell_{\gamma_2}(X))\tr(\rho(\gamma_1^{k_1}))\tr(\rho(\gamma_2^{k_2})).
\end{split}
\end{equation}
Define the projection $\pi:X_n\to X$. Recall $\rho = \mathrm{std}_n \circ \varphi_n: \Gamma \rightarrow S_n$. A closed geodesic $\delta$ in $X_n$ corresponds to a closed geodesic $\pi(\delta)$ in $X$ with the same length. Moreover, let $\gamma_0\in \mathcal{P}(X)$, then the number of closed geodesics $\delta=\delta_0^m$, $\delta_0 \in \calP(X_n)$, mapping to $\gamma_0^k$ is given by
\begin{equation*}
     \# \{\delta_0 \in \calP(X_{n}) \mid  \pi(\delta_0^m)=\gamma_0^k\}=\#\{\text{length $k/m$ orbits in the action $\varphi_n(\gamma_0)\in S_n$ }\} .
\end{equation*}
To see this, suppose $\delta$ starts at $y_0\in X_n$ and $\pi(y_0)=x_0$. Then $\delta$ goes through $k/m$ copies of $\mathbb{H}$ in the identification \eqref{eq:cover-q}. Now \eqref{eq:twisted_selberg_trace} follows from
\begin{equation*}
\begin{split}
    \Tr(\rho(\gamma_0^k))=\#\mathrm{Fix}(\varphi_n(\gamma_0^k))&=\sum_{m|k}\sum_{\substack{\text{ $k/m$ orbits in the}\\ \text{action $\varphi_n(\gamma_0)\in S_n$ }}} \frac{k}{m}\\
    &=\sum_{m|k}\sum_{\substack{\delta_0\in \calP(X_{\rho}), \\ \pi(\delta_0^m)=\gamma_0^k}} \frac{\ell_{\delta_0}(X_n)}{\ell_{\gamma_0}(X)}\\
    &=\sum_{\delta_0^m\in \pi^{-1}(\gamma_0^k)}\frac{\ell_{\delta_0}(X_n)}{\ell_{\gamma_0}(X)}.
\end{split}
\end{equation*}

2. By \eqref{eq:compare_length}, we know $k_{\ell}|\gamma_{\ell}|\leq CL$. We may assume $k_{\ell}|\gamma_{\ell}|\leq p/2$ by taking $p\geq 2CL$.
By Lemma~\ref{lem:fixed_point_estimte}, we have \eqref{eq:eigenvalue_condition} with remainder
\begin{equation*}
\begin{split}
     &\left|\sum_{k_{1}, k_2} \sum_{\ell_{\gamma_1}(X),\ell_{\gamma_2}(X)\leq L}  \frac{\ell_{\gamma_1}(X)}{2\sinh \frac{k_{1}\ell_{\gamma_1}(X)}{2}}\frac{\ell_{\gamma_2}(X)}{2\sinh \frac{k_{2}\ell_{\gamma_2}(X)}{2}}F(k_1\ell_{\gamma_1}(X))F(k_2\ell_{\gamma_2}(X)) \right|(Cp)^{Cp}n^{-p} \\
     &\leq e^{2p} (Cp)^{Cp}n^{-p}\|F\|_{(0,L]}^2 .\qedhere
\end{split}
\end{equation*}
\end{proof}

Finally, we apply Theorem~\ref{thm:lambda-g} to conclude Theorem~\ref{thm:cover}.

\begin{proof}[Proof of Theorem~\ref{thm:cover}]

We apply Theorem~\ref{thm:lambda-g} to the random cover model. The condition \eqref{eq:sys-2} is obvious since $\injrad(X_n)\geq \injrad(X)$. Since Proposition~\ref{prop:rational_function} verifies \eqref{eq:eigenvalue_condition}, we can apply Theorem~\ref{thm:lambda-g} so that with probability at least $1-n^{-1/10}$,
\begin{equation}\label{eq:improve-thm2}
     \left|N_{X_n}(\Lambda)-\frac{|X_n|}{2\pi} \int_0^{\sqrt{\Lambda-1/4}} r\tanh (\pi r) dr\right|\leq C n^{1-\alpha} \Lambda^2 , \quad \Lambda\in [1/4,\infty).
\end{equation}
We also have the following uniform estimate from \cite[Corollary 3.6]{MrSt14}: 
\begin{equation}\label{eq:unif-thm2}
     \left|N_{X_n}(\Lambda)-\frac{|X_n|}{2\pi} \int_0^{\sqrt{\Lambda-1/4}} r\tanh (\pi r) dr\right| \leq Cn\Lambda^{1/2}.
\end{equation}
Now \eqref{eq:weyl} follows from interpolating \eqref{eq:improve-thm2} with \eqref{eq:unif-thm2}. Finally, \eqref{eq:compare} follows from \eqref{eq:weyl}.
\end{proof}

\printbibliography

\end{document}